\documentclass[11pt,a4paper,reqno,english]{amsart}
\usepackage[T1]{fontenc}
\usepackage{amsmath}
\usepackage{amsthm}
\usepackage{amsfonts}
\usepackage{amssymb}
\usepackage{graphicx}
\usepackage{amsbsy}
\usepackage{eucal}
\usepackage{amsbsy}
\usepackage{mathrsfs}
\usepackage{bbm}

\newcommand{\vc}[1]{\overrightarrow{#1}}

\newcommand{\be}{\begin{equation}}
\newcommand{\ee}{\end{equation}}

\def\R{\mathbb{R}}

\def\e1{\vc{e}_1}
\def\e2{\vc{e}_2}
\def\e3{\vc{e}_3}
\def\dis{\displaystyle}

\catcode`@=11
\def\section{\@startsection{section}{1}%
  \z@{1.5\linespacing\@plus\linespacing}{.5\linespacing}%
  {\normalfont\bfseries\large\centering}}
\catcode`@=12

\newtheorem{theorem}{Theorem}[section]

\newtheorem{lemma}[theorem]{Lemma}

\newtheorem{proposition}[theorem]{Proposition}

\theoremstyle{remark}

\numberwithin{equation}{section}
\newcommand{\bea}{\begin{eqnarray}}
\newcommand{\eea}{\end{eqnarray}}
\newcommand{\bee}{\begin{eqnarray*}}
\newcommand{\eee}{\end{eqnarray*}}

\def\na{\nabla}

\def\e{\varepsilon}

\def\NN{\mathbb{N}}

\def\RR{\mathbb{R}}

\def\ni{\noindent}
\def\bs{\bigskip}

\def\eps{\varepsilon}

\def\bar#1{{\overline #1}}

\catcode`@=11
\def\supess{\mathop{\operator@font Sup\,ess}}
\catcode`@=12

\title[Stability of isotropic steady states for relativistic Vlasov-Poisson]{Stability of isotropic steady states for the relativistic Vlasov-Poisson system}
\author{Cyril Rigault}
\address{IRMAR, Universit\'e Rennes 1, France}
\email{cyril.rigault@univ-rennes1.fr}

\begin{document}



\begin{abstract} In this work, we study the orbital stability of stationary solutions to the relativistic Vlasov-Manev system. This system is a kinetic model describing the evolution of a stellar system subject to its own gravity with some relativistic corrections. For this system, the orbital stability was proved for isotropic models constructed as minimizers of the Hamiltonian under a subcritical condition. We obtain here this stability for all isotropic models by a non-variationnal approach. We use  here a new method developed in \cite{LMR-inv} for the classical Vlasov-Poisson system. We derive the stability from the monotonicity of the Hamiltonian under suitable generalized symmetric rearrangements and from a Antonov type coercivity property. We overcome here two new difficulties : the first one is the  {\it a priori} non-continuity of the potentials, from which a greater control of the rearrangements is necessary. The second difficulty is related to the homogeneity breaking which does not give the boundedness of the kinetic energy. Indeed, in this paper, we does not suppose any subcritical condition satisfied by the steady states.

\end{abstract}

\maketitle

\bigskip

\section{Introduction and main results} 
\subsection{Introduction to the relativistic Vlasov-Poisson system} 
The relativistic Vlasov-Poisson system in dimension three reads:
\be \left\{\begin{array}{l}
 \dis \partial_t f + \frac{v}{\sqrt{1+|v|^2}}\cdot \na_x f -\na_x \phi_f \cdot \na_v f=0, \ \ \ \RR_+\times \RR^3\times\RR^3,\\ \\
 f(t=0,x,v)=f_0(x,v)\geq 0,
\end{array}\right. \label{vpr}\ee 
where the gravitational Poisson field $\phi_f$ is defined for all $t\in\RR^+$ by
\be \left\{ \begin{array}{l}
 \dis \Delta \phi_f (t,x)= \rho_f (t,x)=\int_{\RR^3} f(t,x,v) dv, \\
 \dis \phi_f(t,x)\rightarrow 0 \mbox{ as } |x|\rightarrow +\infty.
 \end{array} \right. \label{def-phi1}\ee
This expression is equivalent to:
 \be \phi_f(x)=-\frac{1}{4\pi |x|}\ast \rho. \label{def-phi2}\ee
 This nonlinear transport system describes the evolution of a stellar system subject to its own gravity with some relativistic corrections.

The Cauchy problem of equations of Vlasov-Poisson type is not yet well understood for weak solutions: although the existence of local weak solutions has been proved , the question of its uniqueness remains unknown (see \cite{BGP}). Moreover, for the relativistic case \eqref{vpr}, in the most recent works about the Cauchy problem of smooth solutions, only radial smooth initial data have given results (see \cite{GS} and \cite{KTZ}). In particular, Glassey and Schaeffer \cite{GS} have proved for radial smooth solutions that a blow-up in finite time $T$ is characterized by the blow-up of the kinetic energy
 $$\int_{\RR^6} \left(\sqrt{|v|^2+1}-1 \right)f(t,x,v)dxdv \rightarrow +\infty \mbox{ as } t \rightarrow T.$$

 There holds for smooth enough solutions of system \eqref{vpr} some important conservation properties: first, the total Hamiltonian
\be \mathcal{H}(f(t))=\int_{\RR^6} \left(\sqrt{|v|^2+1}-1 \right)f(t,x,v)dxdv-\frac{1}{2}\int_{\RR^3}\left|\na  \phi_f(t,x)\right|^2 dx, \label{def-H}\ee
is preserved in time and second we have the conservation of all Casimir functions: for all $\beta \in \mathcal{C}^1(\RR_+,\RR_+)$ such that $\beta(0)=0$,
\be \int_{\RR^6} \beta(f(t,x,v))dx dv = \int_{\RR^6} \beta(f_0(x,v))dx dv. \label{Casimir}\ee
The property \eqref{Casimir} is equivalent to the equimeasurability:
\be \forall t\geq 0, \ \ \ \mu_{f(t)}=\mu_{f_0}, \label{equi} \ee
where the distribution function $\mu_f$ is defined by
\be \forall s\geq 0, \ \ \mu_f(s)=\mbox{meas} \{(x,v)\in \RR^6, \ f(x,v)>s\}. \label{def-mu}\ee
Remark that in particular, the $L^p$ norms of $f$ are conserved.

In this paper, we will consider weak solutions to \eqref{vpr} in the natural energy space
\be \mathcal{E}_p=\left\{ f\geq 0 \mbox{ with } f\in L^1(\RR^6)\cap L^p(\RR^6) \mbox{ and } \sqrt{1+|v|^2}f \in L^1(\RR^6) \right\}. \label{def-space} \ee

For all $f_0\in\mathcal{E}_p$, from classical kinetic Cauchy theory, the system \eqref{vpr} admits a local renormalized solution $f(t)$ in the sens of Diperna-Lions \cite{DL1,DL2}. This solution satisfies \eqref{Casimir} and \eqref{equi} but the conservation of Hamiltonian does not occur in general: it only remains
\be \forall t\geq 0, \ \ \ \mathcal{H}(f(t))\leq\mathcal{H}(f_0). \label{decroiH}\ee
Note that, from this nonincreasing property and from \eqref{Casimir}, the kinetic energy will be bounded if $f_0$ satisfies a certain subcritical condition
\be C_p \left\| f_0 \right\|_{L^1}^\frac{2p-3}{3(p-1)}  \left\| f_0 \right\|_{L^p}^\frac{p}{3(p-1)} <1. \label{subcrit}\ee
Indeed, from Hardy-Littlewood-Sobolev inequality, the potential energy satisfies
\be \textrm{for all } f\in \mathcal{E}_p, \ \ \frac{1}{2} \int_{\RR^3}\left|\na  \phi_f \right|^2 dx\leq C_p \left\| f \right\|_{L^1}^\frac{2p-3}{3(p-1)}  \left\| f \right\|_{L^p}^\frac{p}{3(p-1)}  \left\| \sqrt{1+|v|^2} f \right\|_{L^1} \label{interpolation}\ee
where $C_p$ is defined as the best corresponding constant and thus the kinetic energy can be controled thanks to 
\be  \mathcal{H}(f) \geq \left(1- C_p \left\| f \right\|_{L^1}^\frac{2p-3}{3(p-1)}  \left\| f \right\|_{L^p}^\frac{p}{3(p-1)} \right) \left\| \sqrt{1+|v|^2} f \right\|_{L^1}-\left\| f \right\|_{L^1}.\label{cont-kin}\ee
Hence we have global existence for the system \eqref{vpr} as soon as $f_0$ satisfies the subcritical condition \eqref{subcrit}.

\subsection{Main results and strategy of the proof} Our main result gives the nonlinear stability of a large class of stationary solutions for the relativistic Vlasov-Poisson equation. We  adapt here some techniques introduced for the classical Vlasov-Poisson equation \cite{LMR-inv}.

Note (see \cite{BFH} for the Vlasov-Poisson system) that the isotropic steady states of the system \eqref{vpr} are the functions of the form
$$Q(x,v)=F(e),$$
where $e$ is the miscroscopic energy given by
$$e(x,v)= \sqrt{|v|^2+1}-1+\phi_Q(x).$$
An important question which has been the subject of several works is the question of nonlinear stability of these steady states when $F$ satisfies
$$\frac{\partial F}{\partial e} <0.$$
For some of these stationary solutions, built as minimizers of suitable functional, the concentration-compactness Lemma (\cite{L1}- \cite{L2}) allows to prove their stability under a subcritical condition of the type \eqref{subcrit}, see \cite{LM2}. Here we get the nonlinear stability of all these stationary solutions in the energy space $\mathcal{E}_p$, defined by \eqref{def-space}, with respect to its natural norm:
$$\|f\|_{\mathcal{E}_p}=\|f\|_{L^1} + \|f\|_{L^p} +\|  \sqrt{1+|v|^2} f\|_{L^1}.$$

\begin{theorem}[Stability of spherical models]\label{thm} Let $Q$ be a spherical, continuous, nonnegative, non zero, compactly supported steady solution to \eqref{vpr}. Assume that $Q$ is a nonincreasing function of its microscopic energy, i.e. there exists a continuous function $F:\RR\rightarrow \RR_+$ such that for all $(x,v)\in \RR^6$
\be Q(x,v)=F\left(\sqrt{|v|^2+1}-1+\phi_Q(x)  \right), \label{exp-Q}\ee
and there exists $e_0<0$ such that $F(e)=0$ for $e \geq e_Q$, $F$ is decreasing and $\mathcal{C}^1$ on $(-\infty,e_Q)$.
Let $p>\frac{3}{2}$. Then $Q$ is orbitally stable in the $\mathcal{E}_p$-norm by the flow \eqref{vpr}: for all $\varepsilon>0$ there exists $\eta>0$ such that the following holds true. Let $f_0 \in \mathcal{E}_p$ be such that 
\be \left\| f_0-Q \right\|_{\mathcal{E}_p}\leq \eta \label{f0-eta}\ee
and let $f(t)$ be a corresponding renormalized solution to \eqref{vpr} on $[0,T)$. Then for all $t\in[0,T)$ there exists a continuous translation shift $z(t)$ such that
 \be \left\| f(t,x,v)-Q(x-z(t),v) \right\|_{\mathcal{E}_p}\leq \varepsilon. \label{stab-eps}\ee
\end{theorem}

In the radial case, since a blow-up in finite time is equivalent to a blow-up of the kinetic energy, Theorem \ref{thm} provides global existence when the initial data is near the (non necessary subcritical) function $Q$. The existence of such $Q$ away from all subcritical class of functions is an open problem.

\bs
\ni
{\bf Strategy of the proof:} 
This proof uses arguments developed in \cite{LMR-inv} for the classical gravitational Vlasov-Poisson system. But here two new main difficulties appear: the first one is that the studied class of functions in the energy space $\mathcal{E}_p$ does not imply the boundedness of the potentials and the second one is related to the breaking homogeneity character of the relativistic problem. Let us give the global strategy of the proof.

For $f\in\mathcal{E}_p$, we know that there exists a nonnegative function on $\RR_+$, $f^\ast$, which is the Schwarz symmetrization of $f$, such that
$$\forall s\geq 0, \ \ \mu_f(s)=\mu_{f^\ast}(s).$$
We can similarly build a rearrangement with respect to the microscopic energy
$$e(x,v)= \sqrt{|v|^2+1}-1+\phi(x),$$
for a non zero potential $\phi$. This rearrangement $f^{\ast \phi}$ is defined by
$$f^{\ast \phi}=f^\ast \circ a_\phi \mbox{ with } a_\phi(\lambda)=\mbox{meas}\{(x,v),\ e(x,v)<\lambda\}$$
and satisfies 
$$\mu_f=\mu_{f^{\ast \phi}}, \ \ \mathcal{H}(f^{\ast \phi_f})\leq \mathcal{H}(f) \mbox{ and } Q=Q^{\ast \phi_Q}.$$
Note that, in the classical Vlasov-Poisson case, this monotonicity of the Hamiltonian was observed in the physics literature \cite{LB,Gar,WZS,Aly}.
Now by defining the functionnal $J$ by
$$J(\phi)=\mathcal{H}(Q^{\ast \phi})+\frac{1}{2}\left\| \nabla \phi_{Q^{\ast \phi}} -\nabla \phi \right\|_{L^2}^2,$$
we shall prove that
\be \mathcal{H}(f)-\mathcal{H}(Q)\geq J(\phi_{f})-J(\phi_Q)+ \int_0^{+\infty} a_{\phi_{f} }^{-1}(s) (f^\ast(s)-Q^\ast(s)) ds.
\label{linkHJ}\ee
Hence we can reduce our problem to the study of the functionnal $J$ which only depends on the potential $\phi$. Moreover, from the P\'olya-Szegö inequality, we get
$$J(\phi^\ast) \leq J(\phi)$$
where $\phi^\ast$ is the Schwarz symmetrization of $\phi$, and therefore the study of $J$ can be restricted to radial modes. We then use the Burchard-Guo compactness result \cite{BG} to get the compactness of $\phi$ from that of $\phi^\ast$. 

Our proof is performed in two steps:

\ni
(i) Proof of the local coercivity of $J$ on radial potentials $\phi$ near $\phi_Q$.

\ni
(ii) Proof of Theorem \ref{thm} via a local compactness proposition thanks to \eqref{linkHJ}.

\ni
This second step will be deduced from \eqref{linkHJ} and equimeasurability arguments derived from the rigidity of the flow. The local coercivity stated in the first step will follow from the Taylor expansion near $\phi_Q$ for radial potentials
$$J(\phi)-J(\phi_Q)=\frac{1}{2}D^2 J(\phi_Q)(\phi-\phi_Q,\phi-\phi_Q)+o \left( \| \na \phi-\na \phi_Q \|_{L^2}^2\right)$$
and from the strict coercivity of the quadratic form $D^2 J(\phi_Q)$. To prove this coercivity, we follow the same lines as \cite{LMR-inv} where a Poincaré-like inequality was proved for the classical Vlasov-Poisson system. This inequality is a generalisation of Antonov type coercivity estimate and is based on a Hörmander approach \cite{H1,H2}. The main new difficulties here, lie in the control of the jacobian $a_\phi$ which can't be bounded in general because of the non boundedness of the potential and the homogeneity breaking. \\

This papers is organized as follows. Section \ref{section-J} deals with the proof of the local coercivity of the functional $J$ near $\phi_Q$ stated in Proposition \ref{prop-coercivity}. By differentiating the functional $J$ at $\phi_Q$ in Subsection \ref{soteJ}, we obtain the second order Taylor expansion stated in Lemma \ref{taylor-J}. Then, in Subsection \ref{poppc}, from Hardy type control argument, we deduce the Proposition \ref{prop-coercivity} for radially symmetric potential and we finally generalize this proposition for all potential by using a compactness argument \cite{BG}.Section \ref{NSQ} is devoted to the proof of Theorem \ref{thm} from the coercivity result stated in Section \ref{section-J}. First, in Subsection \ref{sub-compact}, we prove the local compactness of local minimizing sequences, Proposition \ref{comp-fonct}. Then, in Subsection \ref{sub-stab}, by using this compactness result combined with a contradiction argument, we finally deduce the orbital stability, Theorem \ref{thm}. The several property about the generalized rearrangement used in these both sections can be found in the completed study of the rearrangement with respect of the microscopic energy stated in Appendix A and B.


\section{Coercivity of the functional $J$}\label{section-J}
The aim of this section is to prove the following proposition \ref{prop-coercivity} below. Our study is based on rearrangements with respect to the microscopic energy. To define these rearrangements, we introduce, for $q>3$,
\be \Phi_q=\left\{ \phi \in L^q(\RR^3) \textrm{ s.t. } \phi\leq 0,\ m(\phi)>0, \ \nabla \phi \in L^2(\RR^3), \ \lim_{|x|\rightarrow+\infty} \phi(x)= 0  \right\} ,\label{def-phiq} \ee
with
\be m(\phi)=\inf_{ x\in\R^3 } \left(1+|x|\right)|\phi(x)|, \label{def-m}\ee
and the norm on $\Phi_q$
\be \left\| \phi \right\|_{\Phi_q} =  \left\|  \nabla \phi \right\|_{L^2(\R^3)} + \left\| \phi \right\|_{L^q(\R^3)}. \ee
The space $\Phi_q$ is a natural space for the potential of distribution functions in $\mathcal{E}_p$, as given by the following lemma.  

\begin{lemma}\label{lem-ouf} Let $f\in\mathcal{E}_p$ non zero with $p>\frac{3}{2}$. Then the potential $\phi_f$ belongs to $\Phi_q$ for all $3<q\leq \frac{3(4p-3)}{p}(<12)$.
\end{lemma}
\begin{proof} From classical interpolation methods and from the Hardy-Littlewood-Sobolev inequality, we have for $f\in\mathcal{E}_p$ with $p>1$
$$\rho_f \in L^{r}(\RR^3), \ 1\leq r\leq \frac{4p-3}{3p-2} \ \mbox{ and } \phi_f\in L^{q}(\RR^3), \ \frac{3}{2}<q\leq \frac{3(4p-3)}{p}. $$
Moreover, for $p>\frac{3}{2}$ the Hardy-Littlewood-Sobolev inequality implies that $\na \phi_f$ belongs to $L^2(\RR^3)$. We prove now that $m(\phi_f)$ defined by \eqref{def-m} is positive. Since $f$ is non zero, the density $\rho_f$ is non zero, too. Hence there exists $R>0$ such that 
$$M:=\int_{|x|<R} \rho_f(x)dx>0.$$
We have then
$$\left| \phi_f(x)\right| = \int_{\RR^3} \frac{\rho_f(y)}{4\pi|x-y|}dy \geq \int_{|y|<R} \frac{\rho_f(y)}{4\pi|x-y|}dy\geq \int_{|y|<R} \frac{\rho_f(y)}{4\pi(|x|+R)}dy.$$
Finally,
$$\left(|x|+R\right) \left| \phi_f(x)\right| \geq M,$$ which concludes the proof of lemma \ref{lem-ouf}.
\end{proof}
Now for $\phi\in \Phi_q$ and $f\in\mathcal{E}_p$ with $p>\frac{3}{2}$, we define the jacobian $a_\phi$ by
\be \forall e<0, \ \ a_\phi(e)=\textrm{meas} \left\{(x,v)\in\R^6 \ : \ \sqrt{|v|^2+1} -1 +\phi(x)<e \right\}, \label{def-jac} \ee
and the rearrangement with respect to the microscopic energy by 
\be f^{\ast \phi}= \left\{ \begin{array} {lcl}
\dis f^\ast \left( a_\phi \left( \sqrt{|v|^2+1} -1 +\phi(x)\right) \right) & \textrm{if} & \dis  \sqrt{|v|^2+1} -1 +\phi(x)<0 \\ \\
0 & \textrm{if} & \dis \sqrt{|v|^2+1} -1 +\phi(x)\geq 0,
\end{array}\right. \label{def-rearrang}\ee
where $f^\ast$ is the Schwarz-symmetrization of $f$ in $\RR^6$. In Appendix \ref{app} we recall some properties about the Schwarz symmetrization and we give all the properties we need about the Jacobian and the rearrangement with respect to the microscopic energy. In particular, $f^{\ast \phi}$ is well defined in $\mathcal{E}_p$ and the function $Q$ defined in Theorem \ref{thm} satisfies $Q=Q^{\ast {\phi_Q}}$.

Using these definitions, we introduce the functional $J$ on $\Phi_q$ defined by
\be J(\phi)=\int_{\R^6} \left(\sqrt{|v|^2+1} -1 +\phi(x) \right) Q^{\ast \phi}dx dv +\frac{1}{2} \left\|  \nabla \phi \right\|_{L^2}^2,  \ee
which is equivalent to
\be J(\phi)=\mathcal{H}(Q^{\ast \phi})+\frac{1}{2}\left\| \nabla \phi_{Q^{\ast \phi}} -\nabla \phi \right\|_{L^2}^2. \ee
We claim now the following Proposition which we prove in the next subsections.

\begin{proposition}[Local coercivity of the functional $J$]\label{prop-coercivity} There exists a constant $\delta_0>0$ such that, for all $q>3$, the following holds true. Let a sequence $\phi_n$ of $\Phi_q$ such that
\be \forall n\in \NN,\     \left\|  \nabla \phi_n -\nabla \phi_Q \right\|_{L^2} \leq \delta_0\ \mbox{ and } \ \lim_{n\rightarrow+\infty}J(\phi_n)\leq J(\phi_Q).  \label{cond-coer} \ee
Then there exists a sequence of translation shifts in space $x_n$ such that
\be  \left\|  \nabla \phi_n -\nabla \phi_Q( \cdot -x_n) \right\|_{L^2}\rightarrow 0 \mbox{ as } n\rightarrow +\infty . \label{conclusion-coer}\ee
\end{proposition}
This coercivity of the functional $J$ near $\phi_Q$ is the first step to prove the stability of $Q$ stated in Theorem  \ref{thm}. To obtain it, on the one hand, we will look for a second order Taylor expansion of $J$ around $\phi_Q$ and, on the second hand, we will control the second derivative of $J$ at $\phi_Q$ thanks to a Poincaré-type inequality.



\subsection{Second order Taylor expansion of $J$ at $\phi_Q$}\label{soteJ}

In order to prove Proposition \ref{prop-coercivity} we give first a Taylor expansion of the functional $J$ near the potential $\phi_Q$.
\begin{lemma}[Taylor expansion of $J$]\label{taylor-J} Let $6<q<12$ and $\phi, \tilde{\phi} \in \Phi_q$. Then the function 
$$\lambda \mapsto J(\phi+\lambda(\tilde{\phi}-\phi)) $$
is twice differentiable on $[0,1]$.\\ 
Moreover, for $\phi$ in $\Phi_q$ radially symmetric, there holds the Taylor expansion near $\phi_Q$: 
\be J(\phi)-J(\phi_Q)= \frac{1}{2} D^2 J(\phi_Q)(\phi-\phi_Q,\phi-\phi_Q) + \eps(\| \nabla \phi -\nabla \phi_Q \|_{L^2}) \label{taylor-exp},\ee
where $\eps(\delta)=\circ(\delta^2)$ as $\delta \rightarrow 0$.\\ 
Finally the second derivative of $J$ at $\phi_Q$ in the direction $h$ is given by:\\
\be D^2 J(\phi_Q)(h,h)=\int_{\RR^3} \left| \nabla h \right|^2 dx - \int_{\RR^6} \left| F'(e_{\phi_Q}(x,v))\right| \left( h(x)-\Pi h (x) \right)^2 dx dv, \label{D2J} \ee
where $e_{\phi_Q}(x,v) = \sqrt{1+|v|^2}-1- \phi_Q(x)$ and $\Pi h$ is the projector defined by
\be \Pi h(x)= \frac {\dis \int_{\RR^3} \left( \left( 1+e_{\phi_Q}(x,v)-\phi_Q(y) \right)_+^2-1\right)_+^\frac{1}{2}  \left( 1+e_{\phi_Q}(x,v)-\phi_Q(y)  \right) h(y) dy }{  \dis \int_{\RR^3} \left( \left( 1+e_{\phi_Q}(x,v)-\phi_Q(y) \right)_+^2-1\right)_+^\frac{1}{2}  \left(  1+e_{\phi_Q}(x,v)-\phi_Q(y)  \right) dy   }. \label{projector}\ee
\end{lemma}
Remark that the function $\Pi h$ can be seen as the projection of $h$ on the functions of the microscopic energy $e_{\phi_Q}(x,v)$. To prove the lemma \ref{taylor-J} we will first prove that $J$ is two times differentiable on $\Phi_q$, then we will evaluate its derivatives on $\phi_Q$ and finally we will control the rest of the expansion for radially symmetric potentials.
\begin{proof} Let $J_0$ be a functionnal on $\Phi_q$, defined by
\be  \ J_0(\phi) = \int_{\R^6} \left(\sqrt{|v|^2+1} -1 +\phi(x) \right) Q^{\ast \phi}dx dv=\int_{\R^6} e_\phi(x,v) Q^{\ast \phi}dx dv. \label{def-J0}\ee
Then $J$ is given by
\be J(\phi)=J_0(\phi) +\frac{1}{2} \left\|  \nabla \phi \right\|_{L^2}^2. \label{J-J0}  \ee
To differentiate $J$, we just have to differentiate $J_0$. Let $\phi, \tilde{\phi} \in \Phi_q$ and $h= \tilde{\phi} - \phi$. We study then the function $\lambda \mapsto J_0(\phi + \lambda h)$ on $[0,1]$.

 \ni
{\em First derivative of $J_0$.}
From the change of variable \eqref{chg-var}, we get
$$J_0(\phi)=\int_{\inf \phi}^0 e Q^\ast(a_\phi(e))a_\phi'(e) de =\int_{\inf \phi}^0 e\left(G \circ a_\phi \right)'(e)de, $$
where $G$ is the $\mathcal{C}^1$ bounded function, with bounded derivative, defined by
\be G(s)=\int_0^s Q^\ast(\sigma)d\sigma.\ee 
Moreover, from the property \eqref{conv-aa-1} in lemma \ref{jac-deriv}, the jacobian $a_\phi (e)$ converges to $0$ as $e\rightarrow \inf \phi$. Thus we have $\left[ e G \circ a_\phi (e) \right]_{\inf \phi}^0=0$ and an integration by parts gives
\be J_0(\phi)=-\int_{-\infty}^0 G \circ a_\phi(e) de. \label{J0-exp}\ee

\bs
\ni
{\em Step 1 : first derivative of $J_0$.}

To differentiate $J_0$ given by \eqref{J0-exp}, we shall use the Lebesgue's derivation theorem. From formula \eqref{formule-deriv-jac} in the Appendix, we have for all $e<0$
\be  \frac{\partial}{\partial \lambda} G(a_{\phi+\lambda h}(e))=Q^\ast (a_{\phi+\lambda h}(e))\frac{\partial}{\partial \lambda} a_{\phi+\lambda h}(e) \ee
with
$$ \frac{\partial}{\partial \lambda} a_{\phi+\lambda h}(e)=-4\pi \int_{\R^3} K\left( e-\phi(x)-\lambda h(x) \right)h(x)dx,$$
and
\be K(\eta)=\left( \left(1+\eta \right)_+^2-1\right)_+^\frac{1}{2} (1+\eta).\ee
Note that the uniform inequality
$$a_{\phi+\lambda h}(e) \geq \frac{4\pi}{3} \int_{\R^3} \left( \left(1+e+\frac{\min\{m(\phi),m(\tilde{\phi})\}}{1+|x|}\right)_+^2-1\right)_+^\frac{3}{2} dx$$
and the compact support of $Q^\ast$ imply that there exists $e_0<0$ such that 
\be \forall e\geq e_0, \ \forall\lambda \in [0,1], \ \ Q^\ast (a_{\phi+\lambda h}(e))=0.\ee
 Moreover, we have for all $e<e_0$, for all $\lambda \in [0,1]$, 
$$\left| \frac{\partial}{\partial \lambda} a_{\phi+\lambda h}(e) \right|  \leq 4\pi \int_{\R^3} K\left( e_0-\phi(x)-\tilde{\phi}(x)\right) \left| h(x) \right|dx, $$
and $Q^\ast (a_{\phi+\lambda h}(e))\leq \|Q\|_\infty$. Finally we obtain, by noting $\phi_\lambda=\phi+\lambda h$
 \be \frac{\partial}{\partial \lambda}J_0(\phi_\lambda)= 4 \pi \int_{-\infty}^0 \int_{\R^3} Q^\ast (a_{\phi_\lambda}(e)) K\left( e-\phi(x)-\lambda h(x) \right) h(x) dx de. \label{deriv1}\ee

\bs
\ni
{\em Step 2 : second derivative of $J_0$.}

We keep the previous notations. An integration by parts with respect to $e$ yields
$$\frac{\partial}{\partial \lambda}J_0(\phi_\lambda )=-\frac{4\pi}{3} \int_{-\infty}^0 \int_{\R^3}  {Q^\ast}' (a_{\phi_\lambda}(e))a'_{\phi_\lambda}(e) \left( \left(1+e-\phi_\lambda(x)\right)_+^2-1\right)_+^\frac{3}{2} h(x) dx de.$$
We perform now the change of variable $s=a_{\phi_\lambda}(e)$, which gives
\be \frac{\partial}{\partial \lambda}J_0(\phi_\lambda )=-\frac{4\pi}{3} \int_0^{L_0} \int_{\R^3}  {Q^\ast}' (s)  \left( \left(1+a_{\phi_\lambda}^{-1}(s)-\phi_\lambda(x)\right)_+^2-1\right)_+^\frac{3}{2} h(x) dx ds,\label{autre-dJ0}\ee
where $L_0$ is the measure of the support of $Q$ and thus satisfies $Supp(Q^\ast)=[0,L_0)$. Define $$g(\lambda,x,s):= \left( \left(1+a_{\phi_\lambda}^{-1}(s)-\phi_\lambda(x)\right)_+^2-1\right)_+^\frac{3}{2}.$$ 
From the first step, for all $\lambda\in [0,1]$ and for all $s\in [0,L_0]$, we have $a_{\phi_\lambda}^{-1}(s)\leq e_0$ and thus the set
$$\{ x\in \RR^3, \, g(\lambda,x,s)\neq 0 \} \subset \{ x\in \RR^3, \, \phi_\lambda(x)\leq e_0\} \subset \{ x\in \RR^3, \, \phi(x)+\tilde{\phi}(x)\leq e_0\}$$
is uniformaly contained in the bounded set $\Omega=\{ x\in \RR^3, \, \phi(x)+\tilde{\phi}(x)\leq e_0\}.$
Moreover, by recalling the notations $h=\phi-\tilde{\phi}$ and 
\be K(\eta)=\left( \left(1+\eta \right)_+^2-1\right)_+^\frac{1}{2} (1+\eta),\label{def-KK}\ee
we have for $(\lambda,x,s)\in [0,1] \times \Omega\times [0,L_0] $
\be \frac{\partial g}{\partial \lambda}(\lambda,x,s)= 3 K\left( a_{\phi_\lambda}^{-1}(s)-\phi(x)-\lambda h(x) \right) \left(-h(x)+ \frac{\partial}{\partial \lambda}a_{\phi_\lambda}^{-1}(s) \right), \ee
where, from lemma \eqref{jac-deriv} in Appendix,
$$\frac{\partial}{\partial \lambda}a_{\phi_\lambda}^{-1}(s)=\frac{\dis  \int_{\Omega} K\left( a_{\phi_\lambda}^{-1}(s)-\phi(x)-\lambda h(x) \right) h(x) dx }{\dis \int_{\Omega} K\left( a_{\phi_\lambda}^{-1}(s)-\phi(x)-\lambda h(x) \right) dx     }.$$
To differentiate \eqref{autre-dJ0} with respect to $\lambda$, we first prove at fixed $s\in(0,L_0)$ 
\be \frac{\partial }{\partial \lambda} \int_{\Omega} g(\lambda,x,s) h(x) dx=\int_{\Omega} \frac{\partial g}{\partial \lambda}(\lambda,x,s) h(x) dx. \label{diff-1x}\ee
Let fixed $s\in(0,L_0)$. We begin by bounding $\frac{\partial}{\partial \lambda}a_{\phi_\lambda}^{-1}(s)$. For all $x\in\RR^3$
\be \phi(x) +\lambda h(x)\geq \phi(x)+\tilde{\phi}(x) \ \mbox{ and } \ a_{\phi_\lambda}^{-1}(s) \leq e_0,\label{borne11} \ee
which provides
\be 0 \leq \int_{\Omega} K\left( a_{\phi_\lambda}^{-1}(s)-\phi(x)-\lambda h(x) \right) h(x) dx \leq  \int_{\Omega} K\left( e_0-\phi(x)-\tilde{\phi}(x) \right) h(x) dx. \label{num11}\ee
Now denote $$\Omega_{\lambda}=\left\{ x\in\RR^3,\ \phi_\lambda(x)< a_{\phi_\lambda}^{-1}\left(\frac{s}{2}\right)\right\}.$$
Then the set $\Omega_{\lambda}$ is included in $\Omega$ and 
$$\mbox{ meas } (\Omega_\lambda)=a_{\phi_\lambda} \circ a_{\phi_\lambda}^{-1}\left(\frac{s}{2}\right)=\frac{s}{2}.$$
Moreover, for all $x\in\Omega_\lambda$, we have
$$a_{\phi_\lambda}^{-1}(s)-\phi_\lambda(x) \geq  a_{\phi_\lambda}^{-1}(s)-a_{\phi_\lambda}^{-1}\left(\frac{s}{2}\right),$$
which, combined with
$$ \forall \eta >0, \ K(\eta)=\left( \eta \left(2+\eta \right)\right)_+^\frac{1}{2} (1+\eta)\geq \sqrt{2\eta},$$
implies
\be \int_{\Omega} K\left( a_{\phi_\lambda}^{-1}(s)-\phi_\lambda(x)\right) dx \geq\frac{s}{\sqrt{2}} \left( a_{\phi_\lambda}^{-1}(s)-a_{\phi_\lambda}^{-1}\left(\frac{s}{2}\right)\right)^\frac{1}{2}. \label{denom11} \ee
We claim that, at fixed $s>0$, there exists a constant $C>0$ such that for all $\lambda\in[0,1]$  
\be  a_{\phi_\lambda}^{-1}(s)-a_{\phi_\lambda}^{-1}\left(\frac{s}{2}\right) \geq C.\label{diff11}\ee
Indeed, assume that the property \eqref{diff11} does not hold, then there exist $\lambda\in[0,1]$ and a sequence $(\lambda_n)$ such that, as $n\rightarrow +\infty$,
$$\lambda_n \rightarrow \lambda \ \mbox{ and } \  a_{\phi_{\lambda_n}}^{-1}(s)-a_{\phi_{\lambda_n}}^{-1}\left(\frac{s}{2}\right)\rightarrow 0.$$
From lemma \ref{jac-deriv} in Appendix, we conclude that 
$$a_{\phi_{\lambda}}^{-1}(s)-a_{\phi_{\lambda}}^{-1}\left(\frac{s}{2}\right)=0,$$
which is not possible since $s>0$ and $a_{\phi_{\lambda}}^{-1}$ is strictly increasing on $\RR_+^\ast$. Finally from the inequalities \eqref{num11} and \eqref{denom11}, for $s\in(0,L_0)$ there exists a constant $C_s>0$ such that for all $\lambda\in[0,1]$,
$$ \frac{\partial}{\partial \lambda}a_{\phi_\lambda}^{-1}(s) \leq C_s.$$
We can thus uniformly bound
$$ \left| \frac{\partial}{\partial \lambda}g \right| \leq  3  \left(1+e_0-\phi(x)-\tilde{\phi}(x)\right)_+^2 \left(-\phi(x)-\tilde{\phi}(x)+ C_s \right),$$
where we used \eqref{borne11} and the fact that $K(\eta)\leq (1+\eta)^2_+$. By noticing that the function $\phi+\tilde{\phi}$ belongs to $L_{loc}^3(\RR^3)$ (since q>3), Lebesgue derivation theorem provides \eqref{diff-1x}.

\bs
Now aim to integrate with respect to $s\in(0,L_0)$. We define for all $s\in(0,L_0)$ and for all $\lambda\in[0,1]$
$$  I(\lambda,s):=\frac{\partial}{\partial \lambda} \int_\Omega g (\lambda,x,s) h(x) dx=\int_\Omega \frac{\partial}{\partial \lambda}g (\lambda,x,s) h(x)dx,$$
which is continuous function of the variable $\lambda$. Remark that
$$  I(\lambda,s)= 3 \int_\Omega K \left( a_{\phi_\lambda}^{-1}(s)-\phi_\lambda(x)\right) \left(-h(x)+ \frac{\partial}{\partial \lambda}a_{\phi_\lambda}^{-1}(s) \right)h(x) dx, $$
where, from Cauchy-Schwarz inequality,
\be \begin{array} {rcl}
\dis \int_\Omega K \left( a_{\phi_\lambda}^{-1}(s)-\phi_\lambda(x)\right) \frac{\partial a_{\phi_\lambda}^{-1}}{\partial \lambda} h (x)dx & = & \dis \frac{ \left( \dis \int_{\R^3} K \left( a_{\phi_\lambda}^{-1}(s)-\phi_\lambda(x)\right) h(x) dx \right)^2 }{\dis \int_{\R^3} K \left( a_{\phi_\lambda}^{-1}(s)-\phi_\lambda(x)\right) dx     } \\
 & \leq &\dis  \int_{\R^3} K \left( a_{\phi_\lambda}^{-1}(s)-\phi_\lambda(x)\right) \left(h(x)\right)^2 dx.
\end{array} \label{CS}\ee
Thus, we have
$$\left|  I(\lambda,s) \right| \leq 3 \int_{\R^3} K \left( a_{\phi_\lambda}^{-1}(s)-\phi_\lambda(x)\right) \left(h(x)\right)^2 dx.$$
Moreover, since  $Q^\ast$ is decreasing from $\|Q\|_{L^\infty}$ to $0$, the function ${Q^\ast}'$ belongs to $L^1(0,L_0)$ and finally, from Lebesgue's derivation theorem, we get
\be \begin{array}{rl}
\dis \frac{\partial^2}{\partial \lambda^2}J_0(\phi_\lambda )= & \dis 4\pi \int_0^{L_0} \int_{\R^3}  {Q^\ast}' (s) K \left( a_{\phi_\lambda}^{-1}(s)-\phi_\lambda(x)\right) (h(x))^2 dx ds \\
& \dis - 4\pi \int_0^{L_0}  {Q^\ast}' (s) \frac{\dis \left( \int_{\R^3} K \left( a_{\phi_\lambda}^{-1}(s)-\phi_\lambda(x)\right) h(x) dx \right)^2}{\dis \int_{\R^3} K \left( a_{\phi_\lambda}^{-1}(s)-\phi_\lambda(x)\right) dx     }ds.
\end{array} \label{deriv2s}\ee
Using the change of variable $e=a_{\phi_\lambda}^{-1}(s)$,  we get

\be\begin{array}{rl}
\dis \frac{\partial^2}{\partial \lambda^2}J_0(\phi_\lambda )= & \dis 4\pi\int_{-\infty}^{0} \int_{\R^3}  {Q^\ast}' (a_{\phi_\lambda}(e))a'_{\phi_\lambda}(e)   K \left( e-\phi_\lambda(x)\right) (h(x))^2 dx de \\
 & \dis - 4\pi \int_{-\infty}^{0}  {Q^\ast}' (a_{\phi_\lambda}(e))a'_{\phi_\lambda}(e) \frac{\dis \left( \int_{\R^3} K \left( e-\phi_\lambda(x)\right) h(x) dx \right)^2}{\dis \int_{\R^3} K \left( e-\phi_\lambda(x)\right) dx     }de.
\end{array} \label{deriv2}\ee

\bs
\ni
{\em Step 3: derivatives of $J$ at $\phi_Q$:} 

Let $\phi \in \Phi_q$ and $h=\phi-\phi_Q$. Then, by \eqref{deriv1},
 $$ D J_0(\phi_Q)(h)= 4 \pi \int_{-\infty}^0 \int_{\R^3} Q^\ast (a_{\phi_Q}(e))K \left( e-\phi_\lambda(x)\right)  h(x) dx de $$
Note that, from lemma \ref{lem-rearrang}, the function $F$ defined by theorem \ref{thm} satisfies 
$$Q^\ast (a_{\phi_Q}(e)) =F(e).$$
Now perform the change of variable $u=\left( \left(1+e-\phi_Q(x)\right)_+^2-1\right)_+^\frac{1}{2} $ with respect to $e$. Then
$$\begin{array} {rcl}
\dis D J_0(\phi_Q)(h) & =  & \dis 4 \pi \int^{+\infty}_0 \int_{\R^3} F\left( \sqrt{1+|v|^2}-1+\phi_Q(x)\right) u^2 h(x) dx du\\
 & = & \dis \int_{\RR^6} Q(x,v)h(x) dx dv.
 \end{array} $$
Hence, from \eqref{J-J0}, 
$$ DJ(\phi_Q)(h)= D J_0(\phi_Q)(h) + \int_{\RR^3} \nabla \phi_Q \cdot \nabla h dx=0, $$
where we used the Poisson equation satisfied by $\phi_Q$.

We now give the explicit expression of the second derivative of $J_0$ at $\phi_Q$. Remark first that $F'(e)={Q^\ast}' (a_{\phi_Q}(e))a'_{\phi_Q}(e)$ and thus from \eqref{deriv2}, 
$$\begin{array}{rl}
\dis D^2 J_0(\phi_Q)(h,h)= & \dis 4\pi\int_{-\infty}^{0} \int_{\R^3}  F'(e)    K \left( e-\phi_\lambda(x)\right) (h(x))^2 dx de \\
& \dis  - 4\pi \int_{-\infty}^{0} F'(e) \frac{\dis  \left( \int_{\R^3}  K \left( e-\phi_\lambda(x)\right) h(x) dx \right)^2}{\dis \int_{\R^3}  K \left( e-\phi_\lambda(x)\right) dx     }de.
\end{array} $$
We apply the change of variable $u=\left( \left(1+e-\phi_Q(x)\right)_+^2-1\right)_+^\frac{1}{2} $ with respect to $e$ to get
$$ D^2 J_0(\phi_Q)(h,h)=\int_{\RR^6} F'\left(e\right)\left(h(x)\right)^2 dx dv -\int_{\RR^6} F'\left(e\right) \Pi h\left(e\right)h(x) dx dv,$$
where $e=e(x,v)=\sqrt{1+|v|^2}-1+\phi_Q(x)$ and $\Pi h$ is the projector on the space of functions depending only on $e(x,v)$, defined by \eqref{projector}. Hence we have 
$$ D^2 J_0(\phi_Q)(h,h)=\int_{\RR^6} F'\left(e(x,v)\right)\left(h(x)-\Pi h\left(e(x,v)\right)\right)^2 dx dv, $$
and the decomposition \eqref{J-J0} provides 
\be D^2 J(\phi_Q)(h,h)=D^2 J_0(\phi_Q)(h,h)+\int_{\RR^3} \left| \nabla h \right|^2 dx,\ee
which concludes the proof of \eqref{D2J}.

\bs
\ni
{\em Step 4: proof of the Taylor expansion \eqref{taylor-exp}:} 

Let $\phi \in \Phi_q$ radially symmetric and $h=\phi-\phi_Q$. We note for $\lambda\in [0,1]$, $\phi_\lambda:=\phi_Q+\lambda h$. Then, using $DJ (\phi_Q)(h)=0$, we have

\ni
$ \dis  J(\phi_Q+h)-J(\phi_Q)= \frac{1}{2} D^2J(\phi_Q)(h,h) $
\be +\|\nabla h\|_{L^2}^2 \int_0^1 (1-\lambda) \left(D^2J_0(\phi_\lambda)-D^2 J_0(\phi_Q) \right) \left( \frac{h}{\|\nabla h\|_{L^2}}, \frac{h}{\|\nabla h\|_{L^2}} \right) d\lambda.\ee
It is sufficient to prove that 
\be \sup_{\lambda \in [0,1]} \sup_{ \|\nabla \hat{h}\|_{L^2}=1} \left| \left(D^2J_0(\phi_\lambda)-D^2 J_0(\phi_Q) \right) \left( \hat{h},\hat{h} \right) \right| \rightarrow 0\label{supsup} \ee
as $\| \na\phi-\na\phi_Q\|_{L^2}\rightarrow 0$ to obtain the Taylor expansion \eqref{taylor-exp}. Note that the functions $\hat{h}$ in \eqref{supsup} are taken to be radially symmetric. In order to prove \eqref{supsup} we argue by contradiction. Let $\eps>0$, $\psi_n\in\Phi_q$, $\hat{h}_n\in\Phi_q$ and $\lambda_n \in[0,1]$ such that
\be \| \na \psi_n-\na \phi_Q \|_{L^2(\RR^3)} \leq \frac{1}{n}, \, \, \ \ \ \| \na \hat{h}_n \|_{L^2(\RR^3)}=1, \ee
and 
\be \left| \left(D^2J_0(\phi_n)-D^2 J_0(\phi_Q) \right) \left( \hat{h}_n,\hat{h}_n \right) \right|>\eps, \label{faux}\ee
where $\phi_n=(1-\lambda_n)\phi_Q+\lambda_n \psi_n$. The sequence $\phi_n$ satisfies
\be \| \na \phi_n-\na \phi_Q \|_{L^2(\RR^3)} \leq \frac{1}{n}.\label{pourL6}\ee
We recall from \eqref{deriv2s} that 
\be D^2J_0(\phi_n )\left( \hat{h}_n,\hat{h}_n \right) =4\pi \int_0^{L_0}{Q^\ast}' (s)\gamma_n(s)  ds \label{D2-hn},\ee
where 
\be \gamma_n(s)=\int_{\R^3}    g_n(x,s)(\hat{h}_n(x))^2 dx- \frac{\dis  \left( \int_{\R^3} g_n(x,s)\hat{h}_n(x) dx \right)^2}{ \dis \int_{\R^3} g_n(x,s)dx     } \label{D-fns},\ee
and
$$g_n(x,s)= \left( \left(1+a_{\phi_n}^{-1}(s)-\phi_n(x)\right)_+^2-1\right)_+^\frac{1}{2} (1+a_{\phi_n}^{-1}(s)-\phi_n(x)). $$
Notice first that the convergence \eqref{pourL6} implies the convergence of $\phi_n$ to $\phi_Q$ in $L^6(\RR^3)$. Thus, from \eqref{conv-aa-1}, we have 
\be a_{\phi_n}^{-1}(s)\rightarrow a_{\phi_Q}^{-1}(s). \label{cavient}\ee
Moreover, we have
$$\left| \phi_n(r) \right| \leq \int_r^{+\infty} \left| \phi_n'(r)\right| dr \leq \| r \phi_n'(r)\|_{L^2(\RR_+)} \left( \int_r^{+\infty} \frac{1}{r^2} dr\right)^\frac{1}{2},$$
where we used the convergence $\phi_n(r)$ to $0$ as $r\rightarrow +\infty$  from the definition of $\Phi_q$. It gives
 \be \forall r\in \RR_+^\ast, \ \ |\phi_n(r) | \leq \frac{\| \na \phi_n \|_{L^2(\RR^3)}}{\sqrt{4\pi r}}\leq\frac{C}{ r^\frac{1}{2}} .\label{boom}\ee
Thus the set of integration in $x$ in the integral \eqref{D-fns} can be restricted to a bounded domain $\Omega$ uniformly with respect to $s\in[0,L_0]$. Indeed, from the increase of $a_{\phi_n}^{-1}$, for all $s\in[0,L_0]$,
\be D_n(s):=\{x\in \RR^3\ : \ \phi_n(x)<a_{\phi_n}^{-1}(s)\} \subset \{x\in \RR^3\ : \ \phi_n(x)<a_{\phi_n}^{-1}(L_0)\}, \ee
and, since $\phi_n \in\Phi_q$ is nonpositive,
\be D_n(s) \subset \left\{x\in \RR^3\ : \ |x|\leq \frac{C^2}{ e_0^2}\right\}=:\Omega\ee
where $e_0=\sup_{n\in\NN} a_{\phi_n}^{-1}(L_0)<0$.\\
Now, from the local compactness of the Sobolev embedding $\dot{H}^1\hookrightarrow L^p_{loc}$ for $1\leq p<6$,  there exists $\hat{h} \in \dot{H}^1_{rad}$ such that, up to a subsequence,
$$ \phi_n \rightarrow \phi_Q \textrm{ and } \hat{h}_n \rightarrow \hat{h} \mbox{ in } L^p(\Omega) \mbox{ as } n\rightarrow +\infty.$$
At fixed $s\in[0,L_0]$, these convergences combined with the convergence \eqref{cavient} provide, on the one hand, the convergence
for all $i\in\{0,1,2\}$, 
$$\hat{h}_n^i \rightarrow \hat{h}^i \mbox{ in } L^2(\Omega_s),$$
and, on the other hand, the convergence of 
$$x \mapsto g_n(x,s)^2=2\eta+5\eta^2+4\eta^3+\eta^4\ \mbox{ with } \eta=\left( a_{\phi_n}^{-1}(s)-\phi_n(x)\right)_+$$
in $L_x^1(\Omega)$ to $g(\cdot,s)^2$ where
$$g(x,s)=\left( \left(1+a_{\phi_Q}^{-1}(s)-\phi_Q(x)\right)_+^2-1\right)_+^\frac{1}{2} (1+a_{\phi_Q}^{-1}(s)-\phi_Q(x)).$$
Thus, for all $i\in\{0,1,2\}$, as $n\rightarrow +\infty$
$$ g_n(x,s)\hat{h}_n^i \rightarrow g(x,s) \hat{h}^i \mbox{ in } L^1_x(\RR^3).$$
The convergence of $\gamma_n(s)$, at fixed $s$, follows. 
\ni
Now, by Cauchy-Schwarz, we have
$$ \left|  \frac{ \dis \left( \int_{\R^3} g_n(x,s)\hat{h}_n(x) dx \right)^2}{ \dis \int_{\R^3} g_n(x,s)dx     } \right| \leq  \int_{\R^3}  g_n(x,s)(\hat{h}_n(x))^2 dx, $$
which provides
$$\begin{array}{rcccl} 
0 &\leq & \gamma_n(s)& \leq & \dis \int_{\R^3}  g_n(x,s)(\hat{h}_n(x))^2 dx\\
 & & & \leq & \dis  \int_{\R^3} g_n(x,L_0)(\hat{h}_n(x))^2 dx.
 \end{array} $$
Thus $\gamma_n$ is uniformly bounded on $[0,L_0]$ and from standard dominated convergence theorem,
\be D^2J_0(\phi_n )\left( \hat{h}_n,\hat{h}_n \right) \rightarrow D^2J_0(\phi_Q )\left( \hat{h},\hat{h} \right) \mbox{ as } n\rightarrow +\infty,\ee
and similarly 
\be D^2J_0(\phi_Q )\left( \hat{h}_n,\hat{h}_n \right) \rightarrow D^2J_0(\phi_Q )\left( \hat{h},\hat{h} \right) \mbox{ as } n\rightarrow +\infty.\ee
These convergences contradict \eqref{faux}, which proves the Taylor expansion \eqref{taylor-exp} and concludes the proof of Lemma \ref{taylor-J}.
\end{proof}

\subsection{Proof of Proposition \ref{prop-coercivity}.} \label{poppc}

We use now the Taylor expansion stated in lemma \ref{taylor-J} to obtain the proposition \ref{prop-coercivity}. In a first step we prove the local coercivity of the functionnal $J_0$ near $\phi_Q$ for radially symmetric potentials by using a Hardy type control, obtained in a second step.  In a third step, we finally pass from radially symmetric modes to general modes by using a compactness argument \cite{BG} which concludes the proof of the proposition \ref{prop-coercivity}. 

\bs
\ni
{\it Step 1: coercivity of the quadratic form $D^2J(\phi_Q )$}

In this step, our aim is to prove that there exists an universal constant $C_0>0$ such that 
\be \forall h \in \dot{H}^1_{rad} \ \ \ \ D^2 J(\phi_Q)(h,h)\geq C_0 \| \na h \|_{L^2}^2.\label{coercivity1}\ee
where the space $$\dot{H}^1_{rad}=\left\{ h \in L^2_{loc}(\RR^3), \mbox{ radially symmetric, s.t. } \na\phi \in L^2(\RR^3) \mbox{ and } \lim_{|x|\rightarrow +\infty}\phi(x)=0\right\} .$$
is a Banch space. We consider the linear operator generated by the Hessian $D^2(\phi_Q)$:
$$ \mathcal{L}h=-\Delta h- \int_{\RR^3} |F'(e)| (h-\Pi h ) dv.$$
Remark that the compactness of the quadratic form $D^2J(\phi_Q )$ on $\dot{H}^1_{rad}$ is given by the previous proof of the Taylor expansion. From the Fredholm alternative, we have only to prove the strict posivity
\be \forall h \in \dot{H}^1_{rad}, \ h\neq 0, \ (\mathcal{L}h,h) > 0.\label{st-positive-on} \ee
to obtain the coercivity \eqref{coercivity1}.
From the Taylor expansion \eqref{D2J}, this inequality can be seen as a Poincar\'e inequality with an explicit constant, and we shall adapt the H\"ormander's proof \cite{H1,H2} to obtain it.

Let us introduce the following operator $T$ defined by
$$Tf (e,r)= \frac{\partial_r f}{r^2 \left( (1+e-\phi_Q(r))^2-1\right)^\frac{1}{2} (1+e-\phi_Q(r))} =  \frac{\partial_r f}{r^2 u \sqrt{1+u^2}},$$
where 
\be u(r,e)=\left( (1+e-\phi_Q(r))_+^2-1\right)_+^\frac{1}{2}.\label{ure}\ee
Recalling that $\phi_Q(r)$ is strictly increasing and that $Supp(F)=[0,e_Q)$, we shall denote the space
$$\mathcal{U}=\{(r,e), \, u>0 \}= \{ (r,e), \, e\in(\phi_Q(0),0), \, r\in (0,r(e)) \} \ \mbox{ with } r(e)=\phi_Q^{-1}(e),$$
Then we define on $\tilde{\mathcal{U} }=\mathcal{U} \cap (0,r(e_Q))\times (\phi_Q(0),e_Q)$, for a given $h \in \dot{H}^1_{rad}$, the function
\be f(r,e)=\int_0^r (h(\tau)-\Pi h (e) )  \left( (1+e-\phi_Q(\tau))^2-1\right)_+^\frac{1}{2} (1+e-\phi_Q(\tau))\tau^2 d\tau.\label{def-fff}\ee
We can differentiate $f$ and get, in particular,
\be Tf=h-\Pi h.\label{Tf-h}\ee
Now let $\eps>0$ and study the behavior of $f(e,r)$ for $r\rightarrow 0$ and $r\rightarrow r(e)$ when $e$ belongs to $(\phi_Q(0)+\eps,-\eps)$. Notice first that, from \eqref{boom}, for all $\tau>0$
 \be \tau^\frac{1}{2}| h(\tau)|\leq \|\na h \|_{L^2},\label{cont-rh}\ee
 and
 \be  \forall e\in(\phi_Q(0)+\eps,-\eps), \ \ \left| \Pi h(e) \right| \leq C_\eps. \label{cont-pih}\ee
These inequalities combined with the continuity of $\phi_Q$ imply the existence of $C_\eps>0$ such that, for all $e\in(\phi_Q(0)+\eps,-\eps)$ and for all $r\in(0,r(e))$,
\be  \left| f(e,r)\right| \leq C_e r^\frac{5}{2}.  \label{bord2}\ee
Moreover the function $f$ satisfies
\be f(e,r(e))=\int_0^{+\infty}(h(\tau)-\Pi h (e) )u(r,e) (1+e-\phi_Q(\tau))\tau^2 d\tau=0, \label{bord11}\ee
where we used $$\Pi h(e) =\frac{\dis \int_0^{+\infty}u(\tau,e) (1+e-\phi_Q(\tau))h(\tau) \tau^2 d\tau}{\dis \int_0^{+\infty}u(\tau,e) (1+e-\phi_Q(\tau)) \tau^2 d\tau}.$$
Hence, using \eqref{cont-rh},\eqref{cont-pih} and the increase of $\phi_Q$, we get for $e\in(\phi_Q(0)+\eps,-\eps)$ and $r\in(0,r(e))$
$$\begin{array}{rcl}
\dis \left| f(r,e)\right| & = & \dis \left| \int_r^{r(e)} (h(\tau)-\Pi h (e) )  \left( (1+e-\phi_Q(\tau))^2-1\right)_+^\frac{1}{2} (1+e-\phi_Q(\tau))\tau^2 d\tau\right|\\
 & \leq & \dis C_\eps \left( (1+e-\phi_Q(r))^2-1\right)_+^\frac{1}{2} (1+e-\phi_Q(r))  \int_r^{r(e)} \tau^\frac{3}{2} d\tau,
 \end{array}$$
in which
$$\left( (1+e-\phi_Q(r))^2-1\right)_+^\frac{1}{2} (1+e-\phi_Q(r)) \leq (e-\phi_Q(r))_+^\frac{1}{2} (2-\phi_Q(0))^\frac{1}{2}(1-\phi_Q(0)),$$
and
$$\int_r^{r(e)} \tau^\frac{3}{2} d\tau \leq C (r(e)-r) \lesssim (e-\phi_Q(r)).$$
since $\phi'_Q(r)\geq \inf \{\phi'_Q(\tau), \tau\in[r(-\eps),r(\phi_Q(0)+\eps)]\}>0$.
Finally there exists $C_\eps>0$ such that for all $e\in(\phi_Q(0)+\eps,-\eps)$ and for all $r\in(0,r(e))$
\be \left| f(e,r) \right|\leq C_\eps (e-\phi_Q(r))^\frac{3}{2}.\label{bord1} \ee 
Now denote $$I(h)=\int_{\RR^6}  |F'(e)| (h-\Pi h)^2 dxdv.$$
First, passing to the spherical coordinates and performing the change of variable $e=\sqrt{1+|v|^2}-1+\phi_Q(r)$, we get from \eqref{Tf-h}
$$\begin{array} {rcl} 
I(h) & = & \dis 16\pi^2\int  |F'(e)| (h-\Pi h)^2 r^2 u(e,r) \sqrt{1+u(e,r)^2} dr de\\
     &  =   & \dis  16\pi^2\int_{\phi_Q(0)}^0  |F'(e)| de \int_0^{r(e)} (h(r)-\Pi h(e)) \partial_r f dr,
\end{array}$$
where $u(r,e)$ is defined by \eqref{ure}. Now, from \eqref{bord1} and \eqref{bord2}, we have
$$(h(r(e))-\Pi h(e))f(r(e),e)=0 \ \mbox{ and } \lim_{r\rightarrow 0} (h(r)-\Pi h(e)) f(r,e)=0,$$
from which an integration by parts gives
$$I(h)\leq -16\pi^2\int_{\tilde{\mathcal{U}} } |F'(e)|  f \partial_r h de dr.$$
This inequality, combined with the identity
$$\rho_Q(r)=\frac{4\pi}{3} \int |F'(e)| u(e,r)^3 de,$$
leads to:
$$\begin{array} {rcl}
I(h) & \leq & \dis (4\pi)^\frac{3}{2} \| \na h \|_{L^2(\RR^3)} \left(   \int_0^{r(e_Q)}  \frac{dr}{r^2} \left( \int_{\phi_Q(r)}^{e_Q}  |F'(e)|  f de  \right)^2   \right)^\frac{1}{2} \\
   & \leq &  \dis (4\pi)^\frac{3}{2} \| \na h \|_{L^2(\RR^3)} \left( \frac{3}{4\pi} \int_0^{r(e_Q)} \frac{\rho_Q(r)}{r^2} dr  \int_{\phi_Q(r)}^{e_Q}   |F'(e)| \frac{f^2}{u(e,r)^3} de               \right)^\frac{1}{2},
\end{array}$$
where we used Cauchy-Schwarz inequalities. Performing the change of variable $u=u(e,r)$ with respect to $e$, we finally obtain
\be I(h)\leq  \| \na h \|_{L^2(\RR^3)} \left( 3 \int   \rho_Q(r)|F'(e)| \frac{f^2}{r^4 u^4 \sqrt{1+u^2}} dx dv   \right)^\frac{1}{2} .\label{max-Ih} \ee
Now we claim the following Hardy type control:
\be I(h) \geq 3 \int   \left( \rho_Q(r) + \frac{\phi'_Q(r)}{r(1+u^2)} \right) |F'(e)| \frac{f^2}{r^4 u^4 \sqrt{1+u^2}} dx dv. \label{Hardy}\ee
Assume \eqref{Hardy}, then \eqref{max-Ih} yields
$$I(h)+ 3 \int    \frac{\phi'_Q(r)}{r(1+u^2)}  |F'(e)| \frac{f^2}{r^4 u^4 \sqrt{1+u^2}} dx dv \leq  \| \na h \|_{L^2(\RR^3)}^2.$$
Thus, letting $\eps \rightarrow 0$ yields $(\mathcal{L}h,h)\geq 0$. Moreover, if $(\mathcal{L}h,h)=0$, then $f=0$ on $\tilde{\mathcal{U}}$, which implies $h(r)=\Pi h(e)$ on $\tilde{\mathcal{U}}$, also $0=(\mathcal{L}h,h)=\| \na h \|_{L^2(\RR^3)}^2$ and finally $h=0$. This concludes the proof of \eqref{st-positive-on}.

\bs
\ni
{\it Step 2: Hardy type control.}

Let us prove now the Hardy type control \eqref{Hardy}. Let $g$ be a given smooth function in $\tilde{\mathcal{U}}$ and $q$ such that $f=qg$. After easy computations, we get
\be (Tf)^2=g^2(Tq)^2+T(q^2gTg)-\frac{T^2g}{g} f^2\geq T(q^2gTg)-\frac{T^2g}{g} f^2. \label{compute}\ee
We take $g(e,r)=r^3 u(e,r)^3$. Then, remarking that
$$\frac{\partial u}{\partial r} = -\phi_Q'(r) \frac{\sqrt{1+u^2}}{u}, $$
we have
$$Tg= \frac{\partial_r g}{r^2 u \sqrt{1+u^2}}=3\frac{u^2}{\sqrt{1+u^2}}-3r \phi_Q'(r),$$
and therefore
$$T^2 g=\frac{-3}{ru\sqrt{1+u^2}}\left( \phi_Q''(r)+\frac{\phi_Q'(r)}{r} \left(2+\frac{1}{1+u^2}\right)\right).$$
Since $\Delta \phi_Q=\rho_Q$, it implies
\be \frac{T^2 g}{g}=- \frac{3}{r^4 u^4 \sqrt{1+u^2}}\left( \rho_Q(r)+\phi_Q'(r) \frac{1}{r(1+u^2)}\right). \ee
Injecting this into \eqref{compute} and integrating on $\tilde{\mathcal{U}}$ yield:
\be\begin{array}{rcl} I(h) & \geq &  \dis 3 \int  \left( \rho_Q(r) + \frac{\phi'_Q(r)}{r(1+u^2)} \right) |F'(e)| \frac{f^2}{r^4 u^4 \sqrt{1+u^2}} dx dv \\
   &  & \dis  +  \int  |F'(e)| T\left( f^2 \frac{Tg}{g} \right) dx dv.
\end{array},\label{HHH}\ee
where we used $q^2gTg=f^2 \frac{Tg}{g}$. Now
$$\int  |F'(e)| T\left( f^2 \frac{Tg}{g} \right) dx dv=16\pi^2 \int_{\phi_Q(0)}^{e_Q}  |F'(e)| de \int_0^{r(e)} \partial_r \left( f^2 \frac{Tg}{g} \right)dr.$$
From \eqref{bord1} and \eqref{bord2}, we can perform an integration by parts and we get
$$\int  |F'(e)| T\left( f^2 \frac{Tg}{g} \right) dx dv=0.$$
We obtain thus the Hardy type control \eqref{Hardy} which concludes the proof of the strict positivity of the Hessian $D^2(\phi_Q)$ on radial modes \eqref{st-positive-on} and therefore the coercivity \eqref{coercivity1}.\\

\ni
{\it Step 3: end of the proof:} Let us first prove the proposition \ref{prop-coercivity} for radially symmetric potential. Let $\phi_n $ a sequence of $\Phi_q$, radially symmetric, such that 
\be \forall n\in \NN,\     \left\|  \nabla \phi_n -\nabla \phi_Q \right\|_{L^2} \leq \delta_1,\ \ \limsup_{n\rightarrow+\infty}J(\phi_n)\leq J(\phi_Q),  \label{cond-coer-rad} \ee
where $\delta_1$ will be defined later. Then we use the Taylor expansion \eqref{taylor-exp} and the coercivity \eqref{coercivity1}, which holds only for radially symmetric potential:
$$J(\phi_n)-J(\phi_Q) \geq \frac{C_0}{2}  \left\|  \nabla \phi_n -\nabla \phi_Q \right\|_{L^2}^2+\eps( \left\|  \nabla \phi_n -\nabla \phi_Q \right\|_{L^2}).$$
Now, since $\eps(\delta)=\circ(\delta^2)$ as $\delta\rightarrow 0$, we can choose $\delta_1$ such that 
$$\forall \delta \in [0,\delta_1], \ \ \eps(\delta) \leq   \frac{C_0}{4} \delta^2.$$
Thus 
$$J(\phi_n)-J(\phi_Q) \geq \frac{C_0}{4}  \left\|  \nabla \phi_n -\nabla \phi_Q \right\|_{L^2}^2,$$
which finally provides, from \eqref{cond-coer-rad},
\be \lim J(\phi_n)=J(\phi_Q) \mbox{ and } \left\|  \nabla \phi_n -\nabla \phi_Q \right\|_{L^2} \rightarrow 0 \mbox{ as } n\rightarrow +\infty \label{conv-rad}.\ee
Let pass now at the general case. We consider a sequence $\phi_n $ of $\Phi_q$. We define then $\phi_n^\ast=-(-\phi_n)^\ast $, the opposite of the Schwarz rearrangement of $(-\phi_n)$. We introduce, for all $n\in \NN$, the potential $\phi_n^\#$ given by 
$$\forall x\in\RR^3, \ \phi_n^\#(x)=\phi_n^\ast\left( \frac{4\pi}{3}|x|^3\right).$$
It is the Schwarz rearrangement of $\phi_n$, defined as a function of $\RR^3$. Then we claim that we have
\be J(\phi_n^\#)-J(\phi_n)=\frac{1}{2}\| \na \phi_n^\# \|^2_{L^2}- \frac{1}{2}\| \na \phi_n \|^2_{L^2}\leq 0, \label{decroi1} \ee
and that there exists a constant $\delta_0>0$ such that
\be \left\|  \nabla \phi_n -\nabla \phi_Q \right\|_{L^2} \leq \delta_0 \mbox{ implies }  \left\|  \nabla \phi_n^\# -\nabla \phi_Q \right\|_{L^2} \leq \delta_1. \label{delta}\ee
Assume first that these claims are true and take the sequence $\phi_n$ such that
\be \forall n\in \NN,\     \left\|  \nabla \phi_n -\nabla \phi_Q \right\|_{L^2} \leq \delta_0,\ \ \lim_{n\rightarrow+\infty}J(\phi_n) \leq J(\phi_Q). \ee
Then, from \eqref{decroi1} and \eqref{delta}, the sequence $\phi_n^\#$ satisfies \eqref{cond-coer-rad} and therefore satisfies \eqref{conv-rad}. Moreover we obtain the equality
$$\lim_{n\rightarrow+\infty}J(\phi_n^\#)= \lim_{n\rightarrow+\infty}J(\phi_n)= J(\phi_Q),$$
which implies from \eqref{decroi1} 
\be \lim_{n\rightarrow+\infty} \| \na \phi_n^\# \|_{L^2}  = \lim_{n\rightarrow+\infty} \| \na \phi_n \|_{L^2}. \label{conv-norm-na} \ee
From \eqref{conv-rad} and \eqref{conv-norm-na}, we are in the equality case of the Polya-Szego inequality. Thanks to Theorem 2 in \cite{BG}, we obtain the following compactness result : there exists a sequence of translation shifts in space, $x_n\in\RR$, such that 
$$\left\|  \nabla \phi_n -\nabla \phi_Q(\cdot-x_n) \right\|_{L^2} \rightarrow 0 \mbox{ as } n\rightarrow +\infty.$$
Let us now prove the claims  \eqref{decroi1} and \eqref{delta}. 
Remark that, from Schwarz rearrangement classical properties, for all $n\in \NN$,  $\phi_n^\#$ belongs to $\Phi_q$ and satisfies the Polya-Szego inequality
\be \| \na \phi_n^\# \|_{L^2}\leq  \| \na \phi_n \|_{L^2} \label{polya}. \ee
Moreover, for all $C^1$ function $\beta$ such that $\beta(0)=0$, one has
$$\int_{\RR^3} \beta(\phi_n(x))dx=\int_{\RR^3} \beta(\phi_n^\#(x))dx,$$
which for $e<0$ and for $\beta(t)= \left( (1+e-t)_+^2-1\right)_+^\frac{3}{2}$ implies $a_{\phi_n}=a_{\phi_n^\ast}$ given by the expression \eqref{formule-jac}. Thus, from \eqref{J0-exp},
$$\begin{array} {rcl}
\dis J(\phi_n)&=&\dis \frac{1}{2}  \| \na \phi_n \|_{L^2}^2-\int_{-\infty}^0 G(a_{\phi_n}(e))de \\ & & \\
& = &\dis  \frac{1}{2}  \| \na \phi_n \|_{L^2}^2-\int_{-\infty}^0 G(a_{\phi_n^\#}(e))de, 
\end{array}$$
which provides directly the claim \eqref{decroi1}. Let us prove the claim \eqref{delta} by a contradiction argument. We assume that there exists a sequence $\psi_n$ in $\dot{H}^1$ such that
\be \left\|  \nabla \psi_n -\nabla \phi_Q \right\|_{L^2} \leq \frac{1}{n} \mbox{ and }  \left\|  \nabla \psi_n^\# -\nabla \phi_Q \right\|_{L^2} > \delta_1. \label{delta2}\ee
The contractivity property of the rearrangement in $L^p$-norms and a Sobolev embedding give
$$\| \psi_n^\#-\phi_Q \|_{L^6} \leq \| \psi_n-\phi_Q \|_{L^6} \leq C \| \na \psi_n-\na \phi_Q \|_{L^2}\leq \frac{C}{n}.$$
Moreover by \eqref{polya} and \eqref{delta2}, we have
\be \| \na \psi_n^\# \|_{L^2}\leq  \| \na \psi_n \|_{L^2} \leq \| \na \phi_Q \|_{L^2} +\frac{1}{n} . \label{na-norm}\ee
Hence $\psi_n^\#$ is bounded in $\dot{H}^1_r$ and the sequence $\na \psi_n^\#$ converges to $\na \phi_Q$ in the $L^2$ weak topology. In fact, from Fatou lemma, the inequality \eqref{na-norm} implies
$$\lim_{n\rightarrow +\infty}\| \na \psi_n^\# \|_{L^2}=\|\na \phi_Q \|_{L^2}.$$
Thus the sequence $ \na \psi_n^\#$ converges to $\na \phi_Q$ in $L^2(\RR^3)$ which contradicts \eqref{delta2}.\\
Notice that, for all $q\in(6,12)$, the space $\Phi_q$ is included in $\dot{H}^1$ and thus $\delta_0$ does not depend on $q$. The proof of Proposition \ref{prop-coercivity} is now complete.


\section{Nonlinear stability of $Q$}\label{NSQ}

We are now ready to prove Theorem \ref{thm} in this section. The proof is based on the two following arguments : i) The local coercivity property of $J$ stated in Proposition \ref{prop-coercivity} which ensures the compactness of the potential field, and ii) The compactness of the whole distribution function in the energy space.

\subsection{Local compactness of the distribution function}\label{sub-compact}
We will prove in this subsection that the Proposition \ref{prop-coercivity} implies the following compactness result:
\begin{proposition}[Local compactness of local minimizing sequences]\label{comp-fonct} Let $p>\frac{3}{2}$. Let $\delta_0>0$ the constant defined in Proposition \ref{prop-coercivity} and let $f_n$ be a sequence of $\mathcal{E}_p$ such that
\be  \| f_n^\ast -Q^\ast \|_{L^1(\R)}\rightarrow 0,   \ \| f_n^\ast -Q^\ast \|_{L^p(\R)}\rightarrow 0,  \ \limsup_{n\rightarrow+\infty} \mathcal{H}(f_n) \leq \mathcal{H}(Q), \label{cond-f1n}  \ee
and
\be    \left\| \na \phi_{f_n}  -\na \phi_Q \right\|_{L^2} < \delta_0. \label{cond-delta0} \ee
Then there exists a translation shift $x_n$ such that
\be \left\|   f_n-Q(\cdot-x_n) \right\|_{\mathcal{E}_p} \rightarrow 0 \mbox{ as } n\rightarrow +\infty. \label{conclusion-comp}\ee
\end{proposition}
\begin{proof}[Proof of Proposition \ref{comp-fonct}] Let $(f_n)$ a sequence of $\mathcal{E}_p$ satisfying \eqref{cond-f1n} and \eqref{cond-delta0}.


\bs
{\em Step 1: Compactness of the potential.} We first remark that, from inequality \eqref{ineg-rearr}, we have
$$\mathcal{H}(f_n)\geq \int_{\R^6} \left(\sqrt{|v|^2+1} -1 +\phi_{f_n}(x) \right) f_n^{\ast \phi_{f_n} }dx dv +\frac{1}{2} \left\|  \nabla \phi_{f_n} \right\|_{L^2}^2,$$
which implies, from the change of variable \eqref{chg-var},
$$\begin{array} {rcl}
\dis \mathcal{H}(f_n)-J(\phi_{f_n}) & \geq & \dis \int_{\R^6} \left(\sqrt{|v|^2+1} -1 +\phi_{f_n}(x) \right) \left( f_n^{\ast \phi_{f_n} }-  Q^{\ast \phi_{f_n} } \right)dx dv \\
  & \geq  & \dis  \int_0^{+\infty} a_{\phi_{f_n} }^{-1}(s) (f_n^\ast(s)-Q^\ast(s)) ds.
\end{array}$$ 
Finally
\be \mathcal{H}(f_n)-H(Q)\geq J(\phi_{f_n})-J(\phi_Q)+ \int_0^{+\infty} a_{\phi_{f_n} }^{-1}(s) (f_n^\ast(s)-Q^\ast(s)) ds.\label{H-J} \ee
Now, since $(f_n)$ satisfies \eqref{cond-f1n} and \eqref{cond-delta0}, the sequence $(f_n)$ is bounded in the energy space $\mathcal{E}_p$. From classical interpolation inequalities, for 
$$q=\frac{3(4p-3)}{p} \in(6,12),$$
the sequence $\phi_{f_n}$ belongs to $\Phi_q$ and is bounded in $L^q(\RR^3)$. Thus, from \eqref{control-jac} we have
\be   \left| a_{\phi_{f_n}}^{-1}(s)\right| \leq C \left( \frac{\dis 1 }{\dis s^\frac{1}{q-3}}+ \frac{\dis 1 }{\dis s^\frac{1}{q}} \right).\label{allez}\ee
Hence we obtain from H\"older inequalities
$$ \left| \int_1^{+\infty} a_{\phi_{f_n} }^{-1}(s) (f_n^\ast(s)-Q^\ast(s)) ds\right| \leq 2C \| f_n^\ast -Q^\ast \|_{L^1(\RR)},$$
and
$$\left| \int_0^1 a_{\phi_{f_n} }^{-1}(s) (f_n^\ast(s)-Q^\ast(s)) ds\right| \leq \|a_{\phi_{f_n} }^{-1}\|_{L^{p'}(0,1)} \| f_n^\ast -Q^\ast \|_{L^p(\RR)},$$
where $p'=\frac{p}{p-1}<3$. Notice that $\frac{p'}{q-3}<1$ which gives, from \eqref{allez}, the boundedness :  
$$\|a_{\phi_{f_n} }^{-1}\|^{p'}_{L^{p'}(0,1)} \leq 2C \int_0^1 \frac{1}{s^\frac{p'}{q-1}} ds=2C\left(1-\frac{p'}{q-1}\right).$$ 
We have then the convergence
$$\int_0^{+\infty} a_{\phi_{f_n} }^{-1}(s) (f_n^\ast(s)-Q^\ast(s)) ds \rightarrow 0$$
as $n\rightarrow +\infty$ and, injecting this in \eqref{H-J}, we conclude from \eqref{cond-f1n} that
\be  \lim_{n\rightarrow +\infty} J(\phi_{f_n}) \leq J(\phi_Q). \ee
Together with the condition \eqref{cond-delta0}, this allows us to apply Proposition \eqref{prop-coercivity} and we conclude that there exists a sequence of translation shifts in space $x_n$ such that
\be \left\|  \nabla \phi_{f_n} -\nabla \phi_Q( \cdot -x_n) \right\|_{L^2} \rightarrow 0 \ \textrm{ as } n \rightarrow +\infty. \label{conv-grad}\ee

\bs

{\em Step 2: Convergence of $f_n\left(\cdot +x_n, \cdot \right)$ in $\mathcal{E}_p$.} To obtain the convergence in the energy space $\mathcal{E}_p$, the method that we chose is very similar with the method developed in \cite{LMR-inv}. We renote $f_n:=f_n\left(\cdot +x_n, \cdot \right)$. We remark first that since $Q=Q^{\ast \phi_Q}$
$$\begin{array} {rcl}
 \dis \left| \int_{\R^6} \left(\sqrt{|v|^2+1} -1 +\phi_Q\right) \left( f_n^{\ast \phi_Q}- Q\right) dx dv \right| & = &  \dis\left| \int_0^\infty a_{\phi_Q}^{-1}(s)\left( f_n^\ast(s)-Q^\ast(s) \right)ds\right| \\ \\
 & \leq &  \|\phi_Q\|_{L^\infty} \left\|f_n^\ast -Q^\ast \right\|_{L^1},
\end{array}$$
and thus
\be  \int_{\R^6} \left(\sqrt{|v|^2+1} -1 +\phi_Q(x) \right) \left( f_n^{\ast \phi_Q}- Q\right) dx dv \rightarrow 0 \label{convast1}\ee
Now, from the inequality \eqref{ineg-rearr}, it implies that
$$\liminf_{n\rightarrow +\infty} \int_{\R^6} \left(\sqrt{|v|^2+1} -1 +\phi_Q(x) \right) \left( f_n- Q\right) dx dv\geq 0.$$
Hence, since
$$\mathcal{H}(f_n)=\mathcal{H}(Q) +\frac{1}{2} \left\| \nabla\phi_{f_n} -\nabla \phi_Q \right\|_{L^2}^2 +\int_{\R^6} \left(\sqrt{|v|^2+1} -1 +\phi_Q(x) \right) \left( f_n- Q\right) dx dv,$$
in which 
$$ \limsup_{n\rightarrow +\infty} \mathcal{H}(f_n) \leq \mathcal{H}(Q) \textrm{ and } \lim_{n\rightarrow +\infty} \left\|  \nabla \phi_{f_n} -\nabla \phi_Q \right\|_{L^2}=0,$$
we obtain, as $n\rightarrow +\infty$,
\be \int_{\R^6} \left(\sqrt{|v|^2+1} -1 +\phi_Q(x) \right) \left( f_n- Q\right) dx dv \rightarrow 0. \label{conv-fnQ}\ee
The two convergences \eqref{convast1} and \eqref{conv-fnQ} yield
\be T_n:=\int_{\R^6} e_Q(x,v)\left( f_n- f_n^{\ast \phi_{Q}}\right) dx dv \rightarrow 0  \mbox{ as } n\rightarrow +\infty, \label{conv-fnfnast}\ee
where $e_Q(x,v):=\sqrt{|v|^2+1} -1 +\phi_Q(x)$.
As in the proof of \eqref{ineg-rearr}, we write $T_n$ in the following equivalent form
$$T_n=\int_{t=0}^{+\infty} dt \left( \int_{S^n_1(t)} e_Q(x,v) dxdv -\int_{S^n_2(t)} e_Q(x,v) dxdv \right),$$
where
$$S^n_1(t)=\{(x,v)\in \RR^6, \ \ f_n^{\ast \phi_Q}(x,v)  \leq t < f_n(x,v) \},$$
$$S^n_2(t)=\{(x,v)\in \RR^6, \ \ f_n(x,v)  \leq t < f_n^{\ast \phi_Q}(x,v) \}.$$
From \eqref{inverse2}, we have 
$$\forall(x,v) \in S_1^n(t),\ \ \  e_Q(x,v)\geq (f_n^\ast \circ a_{\phi_Q})^{-1}(t).$$
Thus
$$T_n\geq \int_{t=0}^{+\infty} dt \left( \int_{S^n_1(t)} (f_n^\ast \circ a_{\phi_Q})^{-1}(t) dxdv -\int_{S^n_2(t)} e_Q(x,v) dxdv \right),$$
and since $\mbox{meas}(S_1^n(t))=\mbox{meas}(S_2^n(t))$ for all $t\in\RR_+$,
$$T_n\geq \int_{t=0}^{+\infty} dt \int_{S^n_2(t)} \left[ (f_n^\ast \circ a_{\phi_Q})^{-1}(t) - e_Q(x,v) \right] dxdv.$$
Remark from \eqref{inverse1}, that the right term is nonnegative and thus, from \eqref{conv-fnfnast}, we get as $n\rightarrow +\infty$
\be A_n:=\int_{t=0}^{+\infty} dt \int_{S^n_2(t)} \left[ (f_n^\ast \circ a_{\phi_Q})^{-1}(t) - e_Q(x,v) \right] dxdv \rightarrow 0 \label{conv-An}\ee
We now claim that this implies
\be B_n:=\int_{t=0}^{+\infty} dt \int_{\Omega^n_2(t)} \left[ (Q^\ast \circ a_{\phi_Q})^{-1}(t) - e_Q(x,v) \right] dxdv \rightarrow 0 \label{conv-Bn} \ee
as $n\rightarrow +\infty$ where
$$\Omega^n_2(t)=\{(x,v)\in \RR^6, \ \ f_n(x,v)  \leq t < Q(x,v) \}.$$
To prove it, we decompose
$$S^n_2=(S^n_2\backslash \Omega^n_2) \cup (S^n_2 \cap \Omega^n_2)\ \mbox{ and } \ \Omega^n_2=(\Omega^n_2\backslash S^n_2) \cup (S^n_2 \cap \Omega^n_2).$$
Thus
\be \begin{array}{rl}
A_n-B_n = & \dis \int_{t=0}^{+\infty} dt \int_{\Omega^n_2(t)\backslash S^n_2(t)} \left[ e_Q(x,v) -(Q^\ast \circ a_{\phi_Q})^{-1}(t)  \right] dxdv \\
                    & \dis + \int_{t=0}^{+\infty} dt \int_{S^n_2(t)\backslash \Omega^n_2(t)} \left[ (f_n^\ast \circ a_{\phi_Q})^{-1}(t) - e_Q(x,v) \right] dxdv \\
                    & \dis + \int_{t=0}^{+\infty} dt \int_{S^n_2(t)\cap \Omega^n_2(t)} \left[ (f_n^\ast \circ a_{\phi_Q})^{-1}(t) - (Q^\ast \circ a_{\phi_Q})^{-1}(t) \right] dxdv
                    \end{array} \label{An-Bn}\ee
Now let us examine each term as $n\rightarrow +\infty$. We first observe that for $g,h \in L^6(\RR^6)$ we have
\be \int_0^{+\infty} \mbox{meas}(\{g< t \leq h \})dt = \int_{\RR^6} ( h-g)_+ dxdv,.\label{pasmal}\ee
Thus we obtain
$$ \int_0^{+\infty} \mbox{meas}(S^n_2(t)\backslash \Omega^n_2(t))dt  \leq  \int_0^{+\infty} \mbox{meas}(\{Q< t \leq  f_n^{\ast \phi_Q} \})dt = \int_{\RR^6} ( f_n^{\ast \phi_Q}-Q)_+ dxdv, $$
which gives, from \eqref{cond-f1n},
$$  \int_0^{+\infty} \mbox{meas}(S^n_2(t)\backslash \Omega^n_2(t))dt  \leq \|f_n^{\ast \phi_Q}-Q\|_{L^1} =|f_n^\ast-Q^\ast|_{L^1}\rightarrow 0, $$
and similarly for $\mbox{meas}(\Omega^n_2(t)\backslash S^n_2(t))$. Using in addition the estimate
$$| e_Q(x,v)| \leq |\phi_Q(0)|, \ \  \left|(f_n^\ast \circ a_{\phi_Q})^{-1}(t)\right| \leq |\phi_Q(0)|,\ \ \left|(Q^\ast \circ a_{\phi_Q})^{-1}(t)\right| \leq |\phi_Q(0)|,$$
we deduce that the first two terms of \eqref{An-Bn} converge to $0$ as $n\rightarrow +\infty$. We now deal with the third term.
Combining the strong $L^1$ convergence in \eqref{cond-f1n}, the monotonicity of $f_n^\ast$ and the continuity of $Q^\ast$, we get
$$\forall e \in (\phi_Q(0),0),\ \ f_n^\ast \circ a_{\phi_Q}(e)\rightarrow Q^\ast \circ a_{\phi_Q}(e) \mbox{ as } n\rightarrow +\infty.$$
Thus for $e \in (\phi_Q(0),0)$ such that $Q^\ast \circ a_{\phi_Q}(e)>t$ we have for $n$ large enough 
$$f_n^\ast \circ a_{\phi_Q}(e)>t,$$
which from the definition of the pseudoinverse $(f_n^\ast \circ a_{\phi_Q})^{-1}$ provides
$$e \leq \liminf_{n\rightarrow +\infty} (f_n^\ast \circ a_{\phi_Q})^{-1}(t).$$
From the definition of $(Q^\ast \circ a_{\phi_Q})^{-1}$, we conlude that
$$\liminf_{n\rightarrow +\infty} (f_n^\ast \circ a_{\phi_Q})^{-1}(t) \geq (Q^\ast \circ a_{\phi_Q})^{-1}(t).$$
We just inject it into the third term of \eqref{An-Bn} to obtain
$$\liminf_{n\rightarrow +\infty} (A_n-B_n) \geq 0.$$
Moreover, from \eqref{conv-An}, from the definition of $\Omega^n_2$ and from \eqref{inverse1}, we have
$$A_n\rightarrow 0 \mbox{ and } B_n\geq 0.$$
We conclude that the convergence \eqref{conv-Bn} holds true: 
\be \int_{t=0}^{+\infty} dt \int_{\{f_n\leq t<Q\}} \left[ (Q^\ast \circ a_{\phi_Q})^{-1}(t) - e_Q(x,v) \right] dxdv \rightarrow 0.\label{enfin1} \ee
Since $e\mapsto F(e)$ is continuous and strictly decreasing with respect to $e=\sqrt{|v|^2+1}-1+\phi_Q(x)$ for $(x,v)\in\{Q(x,v)>0\}$, we have
$$t<Q(x,v) \mbox{ implies } Q^\ast \circ a_{\phi_Q})^{-1}(t) - e_Q(x,v)>0.$$
Thus, up to a subsequence,
$$ \mbox{ for a.e. } (t,x,v) \in \RR_+^\ast \times \RR^6,\ \ \mathbf{1}_{\{f_n\leq t<Q\}} \rightarrow 0, \mbox{ as } n \rightarrow +\infty.$$
Since $\mathbf{1}_{\{f_n\leq t<Q\}}\leq \mathbf{1}_{\{ t<Q\}}$ and 
$$ \int_{t=0}^{+\infty} dt \int_{\RR^6} \mathbf{1}_{\{t<Q\}}  dxdvdt =\|Q\|_{L^1}<+\infty, $$
we may apply the dominated convergence theorem to get:
$$ \int_{t=0}^{+\infty} dt \int_{\RR^6} \mathbf{1}_{\{f_n\leq t<Q\}}  dxdvdt\rightarrow 0 \mbox{ as } n \rightarrow +\infty,$$
which, from \eqref{pasmal} is equivalent to
\be \int_{\RR^6} (Q-f_n)_+ dxdv \rightarrow 0 \mbox{ as } n \rightarrow +\infty.\label{presque}\ee
Now we write
$$\begin{array}{rcl}
\dis \int_{\RR^6} (f_n-Q)_+ dxdv & \leq & \dis \int_{\RR^6} (f_n-f_n^{\ast \phi_Q})_+ dxdv + \int_{\RR^6} (f_n^{\ast \phi_Q}-Q)_+ dxdv \\
& \leq & \dis \int_0^{+\infty} \mbox{meas}(\{f_n^{\ast \phi_Q} \leq t<f_n\}) dt + \| f_n^{\ast \phi_Q}-Q \|_{L^1},
\end{array} $$
where, from the equimeasurability of $f_n$ and $f_n^{\ast \phi_Q}$, we have
$$\begin{array}{rcl}
\dis \int_0^{+\infty} \mbox{meas}(\{f_n^{\ast \phi_Q} \leq t<f_n\}) dt & = & \dis \int_0^{+\infty} \mbox{meas}(\{f_n \leq t<f_n^{\ast \phi_Q}\}) dt \\
 & & = \dis \int_{\RR^6} (f_n^{\ast \phi_Q}-f_n)_+ dxdv \\
 & \leq & \dis  \int_{\RR^6} (Q-f_n)_+ dxdv  + \int_{\RR^6} (f_n^{\ast \phi_Q}-Q)_+ dxdv.
\end{array}$$
We finally get
$$ \int_{\RR^6} (f_n-Q)_+ dxdv \leq \int_{\RR^6} (Q-f_n)_+ dxdv + 2 \| f_n^{\ast \phi_Q}-Q \|_{L^1}.$$
Now, since $\| f_n^{\ast \phi_Q}-Q \|_{L^1}=\|f_n^\ast-Q^\ast\|_{L^1} \rightarrow 0$, we obtain from \eqref{presque} the $L^1$ convergence
\be \| f_n-Q \|_{L_1(\RR^6)} \rightarrow 0 \mbox{ as } n \rightarrow +\infty. \label{L1conv} \ee
Now, from the convergence of $\na \phi_{f_n}$ to $\na \phi_Q$ in $L^2(\RR^3)$ and from \eqref{cond-f1n} we have
\be \liminf_{n\rightarrow +\infty} \int_{\R^6} \left(\sqrt{|v|^2+1} -1 \right) \left( f_n- Q\right) dx dv \leq 0.\label{oups} \ee
But $f_n$ converges almost everywhere in $\RR^6$ to $Q$, thus we have an equality in \eqref{oups} and thus the convergence is strong:
\be \int_{\R^6} \left(\sqrt{|v|^2+1} \right) \left| f_n- Q\right| dx dv\rightarrow 0 \mbox{ as } n \rightarrow +\infty. \label{kinconv} \ee
Similarly, we remark that
$$\| f_n \|_{L^p(\RR^6)}= \| f_n^\ast  \|_{L^p(\RR_+)} \rightarrow \| Q^\ast  \|_{L^p(\RR_+)} =\| Q \|_{L^p(\RR^6)},$$
when $n\rightarrow +\infty$ and we obtain the strong $L^p$ convergence and the proof of Proposition \ref{comp-fonct} is completed.

\end{proof}


\subsection{Stability from the local compactness}\label{sub-stab}

Let $p>\frac{3}{2}$. From the Hardy-Littlewood-Sobolev inequality and H\"older inequalities, we have from the classical interpolation's inequality:
\be \left\| \nabla \phi_{f}-\nabla \phi_g \right\|_{L^2}\leq  K \left\| \rho_f -\rho_g \right\|_{L^\frac{6}{5}} \leq  K\left\| f-g \right\|_{L^1}^\frac{2p-3}{6(p-1)}  \left\| f-g \right\|_{L^p}^\frac{p}{6(p-1)}  \left\| |v| (f-g) \right\|_{L^1}^\frac{1}{2}.\label{inter} \ee
Thus for all $f\in\mathcal{E}_p$
\be  \left\|   f - Q \right\|_{\mathcal{E}_p}<\eps \mbox{ implies } \left\| \nabla \phi_{f}-\nabla \phi_Q \right\|_{L^2} < K \eps. \label{cool1} \ee 
Let fixed $\eps_0>0$ such that 
$$K \eps_0 < \frac{\delta_0}{2},$$
where $K$ is the constant in \eqref{cool1} and $\delta_0$ is defined by Proposition \ref{comp-fonct}. Let $\eps \in (0,\eps_0)$. Then an equivalent reformulation of Proposition \ref{comp-fonct} is the following: there exists 
 $$0<\eta<\eps_0$$
such that the following sentence holds true: if $f\in \mathcal{E}_p$ is such that
\be  \| f^\ast -Q^\ast \|_{L^1(\R)}\leq \eta,   \ \| f^\ast -Q^\ast \|_{L^p(\R)} \leq \eta,  \  \mathcal{H}(f) \leq \mathcal{H}(Q)+\eta, \label{yo1}  \ee
and
\be    \left\| \na \phi_{f}  -\na \phi_Q \right\|_{L^2} < \delta_0, \label{yo2} \ee
then there exists a translation shift $y\in\RR^3$ such that
\be \left\|   f-Q(\cdot-y) \right\|_{\mathcal{E}_p} < \eps. \label{yo3}\ee
Remark that the assumption \eqref{yo2} can be replaced by
\be  \inf_{z\in\RR^3} \left\| \na \phi_{f}(\cdot+z)  -\na \phi_Q \right\|_{L^2} < \delta_0. \label{yo2bis} \ee
Indeed, on the first hand, if the condition \eqref{yo2} is satisfied, the condition \eqref{yo2bis} is satisfied too. In the other hand, if a function $f$ satisfies \eqref{yo1} and \eqref{yo2bis}, then there exists $z\in\RR^3$ such that $\tilde{f}=f(\cdot+z,\cdot)$ satisfies \eqref{yo2}. But $\tilde{f}$ satisfies \eqref{yo1} too. Thus we have \eqref{yo3} for $\tilde{f}$ and also for $f$.

Now we prove theorem \ref{thm}. Let $f_0 \in\mathcal{E}_p$ such that 
\be \left\|   f _0- Q \right\|_{\mathcal{E}_p}<\eta.\label{init}\ee
Let $f(t)\in \mathcal{F}( [0,T),\mathcal{E}_p)$ a corresponding renormalized solution to \eqref{vpr} as stated in \cite{DL1,DL2}. We want to show that $f(t)$ satisfies \eqref{yo3} for all $t\in[0,T)$. 

Let us first prove that $f(t)$ satisfies \eqref{yo1} for all $t$. From \eqref{init}, we have 
$$\| f_0 -Q \|_{L^1(\RR^6)}\leq \eta,   \ \| f_0 -Q \|_{L^p(\RR^6)} \leq \eta,  \  \mathcal{H}(f_0) \leq \mathcal{H}(Q)+\eta, $$
and from the property of contraction of  symmetric rearrangement: 
$$\| f^\ast -Q^\ast \|_{L^p(\RR_+)} \leq \| f -Q \|_{L^p(\RR^6)},$$
we deduce that $f_0$ satisfies \eqref{yo1}. By conservation of the flow, we have 
$$\forall t\in (0,T), \ \ f(t)^\ast=f_0^\ast \mbox{ and } \mathcal{H}(f(t))\leq \mathcal{H}(f_0).$$
We conclude that \eqref{yo1} is satisfied for all $t\in(0,T)$.

Let us now prove that \eqref{yo2bis} is satisfied for all $t\in(0,T)$. At $t=0$, since $\eta<\eps_0$, from \eqref{cool1}, we have
$$\left\| \nabla \phi_{f_0}-\nabla \phi_Q \right\|_{L^2}<\frac{\delta_0}{2}.$$
Moreover, from the regularity of the flow, the potential satisfies 
$$t\mapsto \na \phi_{f(t)} \in \mathcal{C}([0,T),L^2(\RR^3)),$$
which provides the continuity on $[0,T)$ of the function
$$\beta : t\mapsto \inf_{z\in\RR^3} \left\| \na \phi_{f(t)}(\cdot+z)  -\na \phi_Q \right\|_{L^2},$$
Indeed for all $z\in\RR^3$, we have the uniform boundedness

\bs
\ni
$\dis \left| \left\| \na \phi_{f(t)}(\cdot+z)  -\na \phi_Q \right\|_{L^2}-\left\| \na \phi_{f(t_0)}(\cdot+z)  -\na \phi_Q \right\|_{L^2}\right|$
$$\ \ \ \ \ \ \ \leq  \left\| \na \phi_{f(t)}(\cdot+z)  -\na \phi_{f(t_0)}(\cdot+z)  \right\|_{L^2}=\left\| \na \phi_{f(t)}  -\na \phi_{f(t_0)} \right\|_{L^2},$$
which provides the continuity of $\beta$.
Now, by a contradiction argument, assume that there exists $t_1>0$ such that  $\beta(t_1)>\delta_0$. Since $\beta(0)<\delta_0/2$ there exists $t_2>0$ such that $\beta(t_2)=3\delta_q/4$. In particular, $f(t_2)$ satisfies \eqref{yo1} and \eqref{yo2bis} and thus there exists $z(t_2)\in \RR^3$ such that
$$\left\|   f(t_2)-Q(\cdot-x(t_2)) \right\|_{\mathcal{E}_p}<\eps<\eps_0.$$
Injecting it into \eqref{cool1}, we conclude that $\beta(t_2)<\delta_0/2$, which contradicts our assumption. The proof of Theorem \eqref{thm} is complete.


\appendix
\section{Jacobian of the microscopic energy} \label{app}
 We give some useful properties of the jacobian $a_\phi$ given by \eqref{def-jac} and the rearrangement with respect to the microscopic energy defined in \eqref{def-rearrang}. We recall that the space $\Phi_q$ is defined for $q>3$ by \eqref{def-phiq}.

We first gather in the following two lemmas some important properties of the jacobian $a_\phi$.

\begin{lemma}[Properties of the Jacobian $a_\phi$]\label{jacobian} Let  $\phi\in \Phi_q$ with $q>3$. We recall that the Jacobian $a_\phi$ is defined as 
\be \forall e<0, \ \ a_\phi(e)=\textrm{meas} \left\{(x,v)\in\R^6 \ : \ \sqrt{|v|^2+1} -1 +\phi(x)<e \right\}. \ee
Then:

(i) We have the explicit formula:
\be \forall e<0, \ \ a_\phi(e)=\frac{4\pi}{3} \int_{\R^3} \left( (1+e-\phi(x))_+^2-1\right)_+^\frac{3}{2} dx. \label{formule-jac} \ee
Notice that $\left( (1+e-\phi(x))_+^2-1\right)_+=\left( (1+e-\phi(x))^2-1\right)\mathbf{1}_{e-\phi(x)>0}(x).$

(ii) $a_\phi$ is $\mathcal{C}^1$ on $(-\infty,0)$ and is a strictly increasing $\mathcal{C}^1$ diffeomorphism from $(\inf \phi, 0)$ onto $\R_+^\ast$, which defines $a_\phi^{-1}$. Moreover there exist two positive constants $C$ and $\tilde{C}$ such that for all $e<0$ and for all $s >0$
\be a_\phi(e) \leq  \frac{C}{|e|^{q-3}}\left(1+\frac{1}{|e|^3} \right)\| \phi\|^q_{L^q}  \ \textrm{ and } \ a_\phi^{-1}(s)\geq -\tilde{C}\left( \frac{\dis \| \phi\|^\frac{q}{q-3}_{L^q} }{\dis s^\frac{1}{q-3}}+ \frac{\dis \| \phi\|_{L^q} }{\dis s^\frac{1}{q}} \right).\label{control-jac}\ee
The quantity $\inf \phi$ is the essential infimum of the measurable function $\phi$. 
\end{lemma}

Let prove these properties.

\begin{proof}[Proof of Lemma \ref{jacobian}] To prove {\it(i)}, remark that for $e<0$
$$\sqrt{|v|^2+1} -1 +\phi(x)<e \Leftrightarrow \left\{ \phi(x)<e \ \textrm{and } |v|< \left( (1+e-\phi(x))_+^2-1\right)_+^\frac{1}{2} \right\}. $$
Since $\dis \lim_{|x|\rightarrow +\infty}\phi=0$, the set $\left\{x\in \R^3\ : \  \phi(x)<e\right\}$ is bounded and since $\phi$ belongs to $L^q(\R^3)$ with $q>3$, we have $\phi \in L^3_{loc}(\R^3)$ which implies
$$x\mapsto \left( (1+e-\phi(x))_+^2-1\right)_+^\frac{1}{2} \ \textrm{ belongs to } L^3(\R^3).$$
Thus for all $e<0$, $a_\phi(e)$ is finite and after passing to the spherical coordinates in velocity we obtain the formula \eqref{formule-jac}.

\bs
\ni
{\em Proof of (ii)}. We define $g(e,x):=\left( (1+e-\phi(x))_+^2-1\right)_+^\frac{3}{2}$. Then, for $e<e_0<0$,
$$0\leq \frac{\partial  g}{\partial e} (e,x)\leq 3\left( (1+e_0-\phi(x))_+^2-1\right)_+^\frac{1}{2}(1+e_0-\phi(x)),$$
which, as function of the variable $x$, belongs to $L^1(\R^3)$. Indeed its support is included in the bounded set $\left\{x\in \R^3\ : \  \phi(x)<e_0\right\}$ and $\phi \in L^2_{loc}(\R^3)$. Hence we may apply the dominated convergence theorem and get that $a_\phi$ is a $\mathcal{C}^1$ function on $\R_-^\ast$, nul on $(-\infty,\inf \phi]$ (if $\inf \phi$ is finite) and strictly increasing on $(\inf \phi, 0)$.\\
We now look for the limit of $a_\phi(e)$ when $e\rightarrow 0$. Since $\phi \in \Phi_q$,
$$ a_\phi(e)\geq \frac{4\pi}{3} \int_{\R^3} \left( \left(1+e+\frac{m(\phi)}{1+|x|}\right)_+^2-1\right)_+^\frac{3}{2} dx \rightarrow +\infty \ \textrm{as} \ e\rightarrow 0.$$
To conclude the proof of {\it(ii)}, let us study the behavior of $a_\phi(e)$ as $e\rightarrow -\infty$ in the case $\inf \phi=-\infty$. We observe, from Hölder inequality, that
 \be a_\phi(e) \leq  \left( \textrm{meas}\{x\in\R^3\ :\ \phi(x)<e\}   \right)^\frac{q-3}{q} \left( \int_{\phi(x)<e} \left( \left(1+e-\phi(x)\right)^2-1\right)^\frac{q}{2} dx\right)^\frac{3}{q},\label{ineg-separe}\ee
where both terms can be controled. The first term satisfies
\be \textrm{meas}\{x\in\R^3\ :\ \phi(x)<e\} \leq \int_{\phi(x)<e} \left( \frac{\phi(x)}{e}\right)^q dx \leq \frac{\| \phi\|_{L^q}^q}{|e|^q},\label{term1-ineg}\ee
and the second term
$$\int_{\phi(x)<e} \left( \left(1+e-\phi(x)\right)^2-1\right)^\frac{q}{2} dx \leq \int_{\phi(x)<e} C \left(1+\left|\phi(x)\right|^q\right) dx \leq C\left(\frac{\| \phi\|_{L^q}^q}{|e|^q}+\| \phi\|^q_{L^q} \right),$$
where we use the inequality \eqref{term1-ineg}. Now, injecting these both inequalities in \eqref{ineg-separe}, we conclude that there exists a constant $C>0$ such that for all $e<0$
\be a_\phi(e) \leq  \frac{C}{|e|^{q-3}}\left(1+\frac{1}{|e|^3} \right)\| \phi\|^q_{L^q}. \ee
Finally the inverse $a_\phi^{-1}$ satisfies for all $s >0$
\be a_\phi^{-1}(s)\geq -\tilde{C}\left( \frac{\dis \| \phi\|^\frac{q}{q-3}_{L^q} }{\dis s^\frac{1}{q-3}}+ \frac{\dis \| \phi\|_{L^q} }{\dis s^\frac{1}{q}} \right),\ee
and the properties {\it(ii)} are proved.
\end{proof}
The first lemma \ref{jacobian} gives first properties about the jacobian $a_\phi$ at fixed $\phi\in \Phi_q$. The regularity of the jacobian with respect to the potential $\phi$ is studied in the following lemma.

\begin{lemma}[Regularity of the Jacobian $a_\phi$ with respect to $\phi$]\label{jac-deriv} Let $q>3$. Then

(iii) Let $(\phi_n)$, $(e_n)$ and $(s_n)$ sequences of respectively $\Phi_q$, $\R_-^\ast$ and $R_+^\ast$. Assume that there exist $\phi\in \Phi_q$, $e\in \RR_-\cup\{-\infty\}$ and $s\in \RR_+^\ast\cup \{+\infty\}$ such that
$$\phi_n \rightarrow \phi \mbox{ in } L^q(\RR^3), \ e_n \rightarrow e \mbox{  and  } s_n \rightarrow s.$$ 
Then by denoting $a_\phi(-\infty)=0$, $a_\phi(0)=+\infty$ and $a_\phi^{-1}(+\infty)=0$, we have
\be a_{\phi_n}(e_n) \rightarrow a_\phi(e) \mbox{ and } a_{\phi_n}^{-1}(s_n) \rightarrow a_\phi^{-1}(s). \label{conv-aa-1}\ee

(iv) Let $\phi, \tilde{\phi} \in \Phi_q$ and let $h=\phi -\tilde{\phi}$. Then the function $(\lambda,e)\mapsto a_{\phi+\lambda h} (e)$ is a $\mathcal{C}^1$ function on $[0,1]\times \R_-^\ast$ and we have 
\be \frac{\partial}{\partial \lambda} a_{\phi+\lambda h}(e)=-4\pi \int_{\R^3} K \left(e-\phi(x)-\lambda h(x)\right)h(x)dx. \label{formule-deriv-jac} \ee
where the function $K$, defined by
$$K(\eta)= \left( \left(1+\eta \right)_+^2-1\right)_+^\frac{1}{2} (1+\eta),$$
is non decreasing and has its support in $\RR_+^\ast$.

(v) With the same notation as {\it (iv)}. Let $s\in \R_+^\ast$. Then the function $\lambda\mapsto a_{\phi+\lambda h}^{-1} (s)$ is a $\mathcal{C}^1$ function on $[0,1]$ and we have 
\be \frac{\partial}{\partial \lambda} a_{\phi+\lambda h}^{-1}(s)  =  \frac{\dis  \int_{\R^3} K\left( a_{\phi+\lambda h}^{-1}(s)-\phi(x)-\lambda h(x) \right) h(x) dx }{ \dis \int_{\R^3}K\left( a_{\phi+\lambda h}^{-1}(s)-\phi(x)-\lambda h(x)\right)dx     }.\label{formule-deriv-jac-1} \ee

\end{lemma}

\begin{proof}
\ni
{\em Proof of (iii):} From the control of the jacobian \eqref{control-jac}, we have directly for $e=-\infty$,
$$ a_{\phi_n}(e_n) \rightarrow 0=a_\phi(-\infty) \ \textrm{ as } n\rightarrow +\infty. $$
Let us now treat the case $e<0$ such that $e\neq -\infty$. Up to a subsequence, we have the convergence
$$\left( \left(1+e_n+\phi_n(x) \right)_+^2-1\right)_+^\frac{3}{2} \rightarrow \left( \left(1+e+\phi (x) \right)_+^2-1\right)_+^\frac{3}{2}$$
almost everywhere in  $\RR^3$. To obtain the convergence in $L^1(\RR^3)$, from a generalized dominated convergence theorem, we have just to prove the following $L^1$-convergence,
$$g_n:=\left( \left(1+e_0+\phi_n \right)_+^2-1\right)_+^\frac{3}{2} \rightarrow \left( \left(1+e_0+\phi \right)_+^2-1\right)_+^\frac{3}{2}=: g$$
where $\dis e_0=\inf_{n\in \NN} e_n <0$. In order to do it, we define the set 
$$\Omega=\left\{x\in\RR^3, \ \phi(x)<\frac{e_0}{2}\right\}.$$
Since $\Omega$ is included in a compact set, the convergence of $g_n$ to $g$ in $L^1(\Omega)$ comes from the convergence of $\phi_n$ to $\phi$ in $L^3_{loc}(\RR^3)$. Out of $\Omega$ we have $g(x)=0$ and, from a similar Hölder inequality as \eqref{ineg-separe}, we get
$$\begin{array}{rcl}
\dis \int_{\RR^3 \backslash \Omega } g_n(x)dx  
& \leq & \dis C \left(\frac{\| \phi_n\|_{L^q}^q}{|e|^q}+\| \phi_n\|^q_{L^q} \right)^\frac{3}{q} \left( \textrm{meas}\{x\in\R^3 : \phi_n(x)<e_0, \phi(x) \geq \frac{e_0}{2} \}   \right)^\frac{q-3}{q} \\ \\
& \leq &\dis \bar{C} \left( \textrm{meas}\{x\in\R^3 :\ \left| \phi_n(x) - \phi (x)\right| \geq \frac{e_0}{2} \}   \right)^\frac{q-3}{q} \\ & & \\
& \leq & \dis \bar{C} \left( \frac{2 \|\phi_n-\phi \|_{L^q}}{e_0} \right)^{q-3}.
\end{array}$$
Thus the sequence $g_n$ converges to $g$ in $L^1(\RR^3\backslash \Omega)$ and finally in $L^1(\RR^3)$. By consequence, we have the convergence $a_{\phi_n}(e_n) \rightarrow a_{\phi}(e)$ for $e\in ]-\infty, 0[$.

Let us now treat the case $e=0$. Let $M>0$ be an arbitrary constant. Denote for all $n\in\NN$ the space $\Omega_n= \left\{ x\in \RR^3, \ |\phi_n(x)| < \frac{m(\phi)}{2(1+|x|)} \right\}$ and let $e_0<0$ such that
$$\frac{4\pi}{3} \int_{\R^3} \left( \left(1+e_0+\frac{m(\phi)}{2(1+|x|)}\right)_+^2-1\right)_+^\frac{3}{2} dx >2M. $$
Remark that this integral is well defined since the inner function is zero out of $B(0,R)$ for 
$$R=\max\left\{0,\frac{m(\phi)}{2|e_0|}-1\right\}.$$
For $n$ large enough, we have $e_n>e_0$ and thus
$$\begin{array} {rcl}
 \dis a_{\phi_n}(e_n) & \geq & \dis \frac{4\pi}{3} \int_{\R^3\backslash \Omega_n} \left( \left(1+e_0+\frac{m(\phi)}{2(1+|x|)}\right)_+^2-1\right)_+^\frac{3}{2} dx \\
      & \geq & \dis 2M - \frac{4\pi}{3} \int_{\Omega_n} \left( \left(1+e_0+\frac{m(\phi)}{2(1+|x|)}\right)_+^2-1\right)_+^\frac{3}{2} dx
\end{array}$$      
To prove that the second term converges to $0$ as $n\rightarrow +\infty$, we remark that the set of integration of this term is $\Omega_n \cap B(0,R)$. Now, from the definitions of $m(\phi)$ and $\Omega_n$, we have for all $x\in\Omega_n$
$$\phi_n(x)-\phi(x)\geq -\frac{m(\phi)}{2(1+|x|)}+\frac{m(\phi)}{1+|x|} \geq \frac{m(\phi)}{2(1+|x|)},$$ 
and thus
$$\|\phi_n-\phi \|^q_{L^q(\RR^3)} \geq \int_{\Omega_n} \left( \frac{m(\phi)}{2(1+|x|)} \right)^q dx \geq \int_{\Omega_n\cap B(0,R)} \left( \frac{m(\phi)}{2(1+|x|)} \right)^q dx. $$
Since $\phi_n$ converges to $\phi$ in $L^q(\RR^3)$, we deduce that the measure of the set $\Omega_n\cap B(0,R)$ converges to $0$, which implies that the integral
$$ \int_{\Omega_n} \left( \left(1+e_0+\frac{m(\phi)}{2(1+|x|)}\right)_+^2-1\right)_+^\frac{3}{2} dx$$
converges to $0$ as $n\rightarrow +\infty$. Hence for $n$ large enough $ a_{\phi_n}(e_n)\geq M$, which concludes the proof of the convergence of $ a_{\phi_n}(e_n)$. 

To prove that $a_{\phi_n}^{-1}(s_n)\rightarrow a_{\phi}^{-1}(s)$, we denote $e_n:= a_{\phi_n}^{-1}(s_n)$. We know from the above result that, if $e_n$ converges to $e\in [0,+\infty]$, then 
$$s_n= a_{\phi_n}(e_n) \rightarrow  a_{\phi}(e).$$ 
Hence, any subsequence of $(e_n)_{n\in\NN}$ converges to $e= a_{\phi}^{-1}(s)$, which gives the convergence of the whole sequence $(e_n)_{n\in\NN}$ to $e= a_{\phi}^{-1}(s)$. The proof of {\it(iii)} is complete.

\bs
\ni
{\em Proof of (iv)}: We recall that $h=\tilde{\phi}-\phi$ with $\phi,\tilde{\phi}\in\Phi_q$ and that $\lambda\in[0,1]$. From the convexity of $\Phi_q$, the function $\phi+\lambda h$ belongs to $\Phi_q$ and $a_{\phi +\lambda h}$ is well-defined. For $e<0$ fixed, we aim to differentiate on $(0,1)$ the function 
$$\lambda \mapsto a_{\phi+\lambda h} (e)=\frac{4\pi}{3} \int_{\R^3} \left( (1+e-\phi(x)-\lambda h(x))_+^2-1\right)_+^\frac{3}{2} dx.$$
First, the set of integration satisfies 
$$\left\{ x\in \R^3\ :\ (\phi +\lambda h)(x)<e\right\} \subset \left\{ x\in \R^3\ :\ \phi (x)<e\right\} \cup \left\{ x\in \R^3\ :\ \tilde{\phi}(x)<e\right\},$$ 
which is included in a compact set of $\R^3$. Now, we have for all $x$ in this set
$$\frac{\partial}{\partial \lambda} \left[\left( (1+e-\phi(x)-\lambda h(x))_+^2-1\right)_+^\frac{3}{2}\right]=3 K\left( e-\phi(x)-\lambda h(x)\right) h(x),$$
where $K(\eta)=\left( \left(1+\eta \right)_+^2-1\right)_+^\frac{1}{2} (1+\eta)$. This derivative can be bound uniformly with respect to $\lambda$ by
$$\begin{array}{rcl} 
\dis K\left( e-\phi(x)-\lambda h(x)\right) \left| h(x) \right| & \leq & \dis K\left( e-\phi(x)-\tilde{\phi}(x)\right)|\phi(x)+\tilde{\phi}(x)|   \\ \\
& \leq &\dis C\left(1+ \left| \phi(x)-\tilde{\phi}(x)\right|^3\right),
 \end{array}$$
which, combined with the fact that $\phi+\tilde{\phi} \in L^3_{loc}(\R^3)$ allows the Lebesgue dominated convergence theorem and thus implies {\it(iv)}.

\bs
\ni
{\em Proof of (v):} Let $s\in \R^\ast_+$. The continuity of $\lambda \mapsto  a_{\phi+\lambda h}^{-1}(s)$ comes directly from {\it(iii)}. Let us differentiate this function on $\lambda_0 \in(0,1)$. Denote $\phi_0=\phi+\lambda_0 h$ and $\phi_\lambda=\phi+\lambda h$ with $\lambda\neq \lambda_0$. Then we rewrite
$$\frac{ a_{\phi_\lambda}^{-1}(s)- a_{\phi_0}^{-1}(s)}{ \lambda-\lambda_0 }=\frac{ a_{\phi_\lambda}^{-1}(s)- a_{\phi_0}^{-1}(s)}{a_{\phi_0}(a_{\phi_\lambda}^{-1}(s))-a_{\phi_0}(a_{\phi_0}^{-1}(s))}   \frac{ a_{\phi_0}(a_{\phi_\lambda}^{-1}(s))-a_{\phi_\lambda}(a_{\phi_\lambda}^{-1}(s))}{ \lambda-\lambda_0 }.$$
Since $a_{\phi_\lambda}^{-1}(s)$ converges to $a_{\phi_0}^{-1}(s)$ as $\lambda\rightarrow \lambda_0$, the first term satisfies 
$$\frac{ a_{\phi_\lambda}^{-1}(s)- a_{\phi_0}^{-1}(s)}{a_{\phi_0}(a_{\phi_\lambda}^{-1}(s))-a_{\phi_0}(a_{\phi_0}^{-1}(s))} \rightarrow \frac{1}{a_{\phi_0}'(a_{\phi_0}^{-1}(s))} \ \textrm{ as } \lambda\rightarrow \lambda_0,$$
and the second term satisfies 
$$ \frac{ a_{\phi_0}(a_{\phi_\lambda}^{-1}(s))-a_{\phi_\lambda}(a_{\phi_\lambda}^{-1}(s))}{ \lambda-\lambda_0 }\rightarrow \left.\frac{\partial}{\partial \lambda} a_{\phi+\lambda h}\left(a_{\phi_0}^{-1}(s)\right) \right|_{\lambda=\lambda_0}.$$
Using \eqref{formule-deriv-jac}, we finally get the expression \eqref{formule-deriv-jac-1}.

\end{proof}

\begin{lemma}[Changes of variables]\label{chgmt} Let $\alpha \in \mathcal{C}^0(\R)\cap L^\infty(\R)$, $G \in L^1(\R_+)$ and for all $(x,v)\in\R^6$
$$e(x,v)=\sqrt{|v|^2+1} -1 +\phi(x).$$
Then
\be \begin{array} {rcl}
 \dis \int_{e(x,v)<0} \alpha \left( e(x,v) \right) G \left( a_\phi \left( e(x,v) \right) \right) dx dv  & = & \dis \int_{\inf \phi}^0 \alpha(e) G (a_\phi(e)) a_\phi'(e) de \\ \\ \dis 
 & = & \dis  \int_0^{+\infty} \alpha( a_\phi^{-1}(s))G(s)ds,
 \end{array} \label{chg-var}\ee
 where $\inf \phi$ is the essential infimum of $\phi$.
\end{lemma}
\begin{proof}
We perform the change of variable $e=\sqrt{|v|^2+1} -1 +\phi(x),$ with respect to the velocity variable $v$ to get \\

$\dis \int_{e(x,v)<0} \alpha \left( e(x,v) \right) G \left( a_\phi \left( e(x,v) \right) \right) dx dv  $
 $$ \begin{array}{cl}
 = &\dis 4 \pi \int_{\R^3} dx \int_{\phi(x)}^0 \alpha(e) G (a_\phi(e)) \left( (1+e-\phi(x))_+^2-1\right)_+^\frac{1}{2}(1+e-\phi(x))de \\ \\
 = & \dis \int_{\inf \phi}^0 \alpha(e) G (a_\phi(e)) a_\phi'(e) de.
 \end{array}$$
 We have then directly \eqref{chg-var} since $a_\phi$ is a $\mathcal{C}^1$-diffeomorphism from $(\inf \phi,0)$ onto $\RR_+^\ast$.
\end{proof}

\section{Rearrangement with respect to the microscopic energy}\label{app-rearr}

We use now this jacobian to define a new rearrangement of any $f\in\mathcal{E}_p$ with respect to the microsopic energy $\sqrt{|v|^2+1} -1 +\phi(x)$, where $\phi$ belongs to $\Phi_q$ given by \eqref{def-phiq}. We first recall some basic properties of the classical Schwarz symmetrization.

\begin{lemma}[Schwarz symmetrization] Let $f\in \mathcal{E}_p$, nonzero, with $p>\frac{3}{2}$. We define the Schwarz symmetrization $f^\ast$ of $f$ on $\RR_+^\ast$ by 
$$\forall t>0, \ \ f^\ast (t)=\inf\{ s\geq 0, \ \ \mu_f(s)\geq t \},$$
where $\mu_f$ is the distribution function of $f$ defined by \eqref{def-mu}. Then $f^\ast$ is the unique nonincreasing function on $\RR_+^\ast$ such that $f$ and $f^\ast$ have the same distribution function
$$\forall s\geq 0, \mu_f(s)=\mu_{f^\ast}(s).$$
Moreover, if $f$ is continuous then $f^\ast$ is continuous. In particular, $Q^\ast$ is continuous. 
\end{lemma}
Now, from this Schwarz symmetrisation, we define a new rearrangement with respect to the microscopic energy.
\begin{lemma}[Symmetric rearrangement of $f$ with respect to the microscopic energy]\label{lem-rearrang} Let $f\in \mathcal{E}_p$, nonzero, with $p>\frac{3}{2}$ and $\phi\in \Phi_q$. Let $f^\ast$ be the Schwarz rearrangement of $f$ in $\R^6$. We recall that the function $f^{\ast\phi}$ is defined by
\be f^{\ast \phi}(x,v)= \left\{ \begin{array} {lcl}
\dis f^\ast \left( a_\phi \left( \sqrt{|v|^2+1} -1 +\phi(x)\right) \right) & \textrm{if} & \dis  \sqrt{|v|^2+1} -1 +\phi(x)<0, \\ \\
0 & \textrm{if} & \dis \sqrt{|v|^2+1} -1 +\phi(x)\geq 0.
\end{array}\right. \ee
Then,

(i) $f^{\ast \phi}$ is equimeasurable with $f$, which means
\be f^{\ast \phi} \in Eq(f)=\{ g\in L^1_+ \cap L^p \textrm{ with } \mu_g=\mu_f\}. \label{def-equi} \ee

(ii) $f^{\ast \phi}$ belongs to $\mathcal{E}_p$ with
\be \left\|  |v| f^{\ast \phi} \right\|_{L^1} \leq  C\left( \left\| \nabla \phi \right\|^2_{L^2}   \left\| f \right\|_{L^1}^\frac{2(2p-3)}{6(p-1)}  \left\| f \right\|_{L^p}^\frac{2p}{6(p-1)} + \left\| f \right\|_{L^1} \right). \label{control-kin}\ee

(iii) A function $Q$ as defined in Theorem \ref{thm} satisfies
$$F=Q^\ast \circ a_{\phi_Q} \textrm{  on  } \R_-^\ast \textrm{  and  }   Q=Q^{\ast \phi_Q} \textrm{  on  } \R^6.$$ 

(iv)  Let $f,g \in \mathcal{E}_p$ satisfying $\mu_g\leq \mu_f$ and $\phiÊ\in \Phi$, then
\be \int_{\R^6} \left(\sqrt{|v|^2+1} -1 +\phi(x) \right) \left( g- f^{\ast \phi}\right) dx dv \geq 0. \label{ineg-rearr} \ee
with equality if and only if $g= f^{\ast \phi}$. In particular
\be \mathcal{H}(g) \geq \mathcal{H}( f^{\ast \phi_g}) +\frac{1}{2} \left\| \nabla\phi_g -\nabla \phi_{f^{\ast\phi_g}} \right\|_{L^2}^2 \geq \mathcal{H}( f^{\ast \phi_g}), \label{ineg-H} \ee
with equality if and only if $g= f^{\ast \phi_g}$.  

\end{lemma}

\begin{proof}[Proof of Lemma \ref{lem-rearrang}]

\ni
{\em (i) Equimeasurability:} Let $\beta\in \mathcal{C}^1(\R_+,\R_+)$ satisfying $\beta(0)=0$. Using the change of variable given by \eqref{chg-var} we have
$$\int_{\R^6} \beta \left( f^{\ast \phi} (x,v) \right) dx dv = \int^{+\infty}_0 \beta \left( f^\ast(s) \right) ds = \int_{\R^6} \beta \left( f (x,v) \right) dx dv,$$
which gives the equimeasurability of $f$ and $f^{\ast \phi}$.\\

\ni
{\em (ii) Control of the kinetic energy:} From the definition of $f^{\ast \phi}$, see \eqref{def-rearrang}, we have
$$\int_{\R^6} \left( \sqrt{|v|^2+1} -1 +\phi(x) \right) f^{\ast \phi} \leq 0,$$
and 
$$ -\int_{\R^6} \phi(x)f^{\ast \phi}(x,v)dx dv  =  \int_{\R^3} \nabla \phi (x)\cdot  \nabla \phi_{f^{\ast \phi}}(x)  dx.$$
Thus, from the Cauchy-Schwartz inequality 
\be \begin{array}{rcl}
 \dis  \int_{\R^6}|v| f^{\ast \phi}  & \leq & \dis \int_{\R^6} \left( \sqrt{|v|^2+1} -1 \right) f^{\ast \phi} +  \left\| f^{\ast\phi} \right\|_{L^1} \\
 \\
 & \leq & \dis \left\| \nabla \phi \right\|_{L^2} \left\| \nabla \phi_{f^{\ast \phi}} \right\|_{L^2} +\left\| f \right\|_{L^1}.
 \end{array} \label{cont1}\ee
Moreover, the interpolation inequality \eqref{interpolation} and the equimesurability of $f$ and $f^{\ast \phi}$ yield
\be \left\| \nabla \phi_{f^{\ast \phi}} \right\|_{L^2} \leq  C \left\| f \right\|_{L^1}^\frac{2p-3}{6(p-1)}  \left\| f \right\|_{L^p}^\frac{p}{6(p-1)}  \left\| |v| f^{\ast \phi}  \right\|_{L^1}^\frac{1}{2}.\label{cont2} \ee
Hence the inequalities \eqref{cont1} and \eqref{cont2} imply
$$ \int_{\R^6}|v| f^{\ast \phi} -K(\phi,f) \left(\int_{\R^6}|v| f^{\ast \phi}\right)^\frac{1}{2} -\left\| f \right\|_{L^1}\leq 0,$$
with 
$$K(\phi,f)=C\left\| \nabla \phi \right\|_{L^2}   \left\| f \right\|_{L^1}^\frac{2p-3}{6(p-1)}  \left\| f \right\|_{L^p}^\frac{p}{6(p-1)}.$$
The control of the kinetic energy \eqref{control-kin} follows.\\

\ni
{\em (iii) The profil Q, fixed point for the rearrangement $f^{\ast \phi_f}$:} From the assumptions on $Q$ in theorem \ref{thm}, the measure of the support of $Q$ is $a_{\phi_Q}(e_Q)$ and the function $F$ is a strictly decreasing $\mathcal{C}^1$-diffeomorphism from $[\min \phi_Q,e_Q]$ onto $[0,\| Q\|_\infty]$. Thus we have for all $e\in [\min \phi_Q,e_0],$
$$\begin{array} {rcl}
\dis \ \ a_{\phi_Q}(e) & = & \dis \textrm{meas} \left\{(x,v)\in \R^6 \ : \ F\left( \sqrt{|v|^2+1} -1 +\phi_Q(x)  \right) >F(e)\right\} \\ \\
 & = & \dis \textrm{meas} \left\{(x,v)\in\RR^6 \ : \ Q(x,v) >F(e)\right\} \\ \\
      & = & \dis \textrm{meas} \left\{s\in\R^\ast_+ \ : \ Q^\ast(s) >F(e)\right\}, \end{array}$$
from the definition of the Schwarz symmetrization $Q^\ast$. We rewrite this equality: for all $a\in [0,\| Q\|_\infty]$
\be a_{\phi_Q}\circ F^{-1}(a)=  \textrm{meas}\left\{s\in\R^\ast_+ \ : \ Q^\ast(s) >a\right\}.\label{truc}\ee
We know that $Q^\ast$ is a decreasing function on $[0,meas(Supp(Q))]=[0,a_{\phi_Q}(e_Q)]$. The equality \eqref{truc} and the continuity of $a_{\phi_Q}\circ F^{-1}$ imply the strict decrease of $Q^\ast$.  Moreover, since the Schwarz symmetrisation conserves the continuity, $Q^\ast$ is continuous on $[0,a_{\phi_Q}(e_Q)]$. We conclude that, for all $a\in[0,\| Q\|_\infty]$,
$$Q^\ast \left( \textrm{meas}\left\{s\in\R^\ast_+ \ : \ Q^\ast(s) >a\right\} \right) =a.$$
We deduce that for all $e\in [\min \phi_Q,e_Q]$, $F(e)=Q^\ast \circ a_{\phi_Q}(e)$, which implies in particular  $0=F(e_Q)=Q^\ast \circ a_{\phi_Q}(e_Q)$. But $Q^\ast \circ a_{\phi_Q}$ is discreasing on $R^\ast_-$ and we also have
\be F=Q^\ast \circ a_{\phi_Q} \textrm{ on } [\min \phi_Q,0).\label{equality}\ee
Let now for all $(x,v)\in \R^6$
$$e(x,v):=\sqrt{|v|^2+1} -1 +\phi_Q(x)\in [\min \phi_Q,+\infty).$$
If $ e(x,v) \leq 0$, from \eqref{equality}, we have $Q(x,v)=F(e(x,v))=Q^{\ast \phi_Q}(x,v)$. If $e(x,v)>0$, from the definition of $Q^{\ast \phi_Q}$, we have $Q^{\ast \phi_Q}(x,v)=0=F(e(x,v))=Q(x,v)$.\\

\ni
{\em (iv) The Hamiltonian of the rearrangement: } Let $f,g\in\mathcal{E}_p$. Then 
$$\begin{array}{rcl}
\dis \left\| \nabla\phi_g -\nabla \phi_{f} \right\|_{L^2}^2 & = &\dis  -\int_{\RR^6} (\phi_g-\phi_f)(g-f)dx dv \\ \\
 & = &\dis  -\int_{\RR^6} \phi_g g dx dv - \int_{\RR^6} \phi_f f dx dv+2 \int_{\RR^6} \phi_g f dx dv,
\end{array}$$
where we use $\int \phi_f g =-\int \na \phi_f \na \phi_g =\int \phi_g f$. Finally we have
$$  \left\| \nabla\phi_g -\nabla \phi_{f} \right\|_{L^2}^2= -2 \int_{\RR^6} \phi_g (g- f) dx dv+\left\| \nabla \phi_{f} \right\|_{L^2}^2-\left\| \nabla\phi_g \right\|_{L^2}^2.$$
Thus we obtain
$$\begin{array}{rcl}
\dis \mathcal{H}(g)-\mathcal{H}( f) & = & \dis \int_{\R^6} \left(\sqrt{|v|^2+1} -1\right) \left( g- f\right) dx dv+\frac{1}{2}\left(  \left\| \nabla \phi_{f} \right\|_{L^2}^2-\left\| \nabla\phi_g \right\|_{L^2}^2 \right) \\ \\
& = & \dis \int_{\R^6} \left(\sqrt{|v|^2+1} -1 +\phi_g \right) \left( g- f \right) dx dv+\frac{1}{2} \left\| \nabla\phi_g -\nabla \phi_{f} \right\|_{L^2}^2,\end{array} $$
where the Hamiltonian is defined by \eqref{def-H}. We apply this equality to $g, f^{\ast \phi_g} \in \mathcal{E}_p$ with $g$ nonzero and $f\in \mathcal{E}_p$ to get

$$\mathcal{H}(g)=\mathcal{H}( f^{\ast \phi_g}) +\frac{1}{2} \left\| \nabla\phi_g -\nabla \phi_{f^{\ast\phi_g}} \right\|_{L^2}^2 +\int_{\R^6} \left(\sqrt{|v|^2+1} -1 +\phi_g \right) \left( g- f^{\ast \phi_g}\right) dx dv .$$
Hence, to prove the inequality \eqref{ineg-H}, it is sufficient to prove \eqref{ineg-rearr}. 

Let $f\in \mathcal{E}_p$, $g\in \mathcal{E}_p$, $\phi \in \Phi_q$ and let $T$ defined by
\be T= \int_{\R^6} \left(\sqrt{|v|^2+1} -1 +\phi(x) \right) \left( g- f^{\ast \phi}\right) dx dv.\ee
The claim $T\geq 0$ is a classical inequality for the rearrangements (see \cite{LL} for the Schwarz rearrangement for example and \cite{LMR-inv} for the new rearrangement). To simplify the notation we define 
$$e(x,v)=\sqrt{|v|^2+1} -1 +\phi(x),$$
and we use the layer cake representation
$$f(x,v)=\int_{\R_+} \mathbf{1}_{t<f(x,y)}dt$$
to find
\be T=\int_0^{+\infty} dt \left(      \int_{S_1(t)} e(x,v) dxdv- \int_{S_2(t)} e(x,v) dxdv\right),\label{decomp-T}\ee
where 
$$S_1(t)=\left\{  (x,v)\in\R^6, \ f^{\ast\phi}(x,v) \leq t< g(x,v)\right\},$$
$$S_2(t)=\left\{  (x,v)\in\R^6, \ g(x,v) \leq t< f^{\ast\phi}(x,v)\right\}.$$
Now, from the properties about the pseudo-inverse of $f^\ast \circ a_\phi$, we deduce that 
$$e(x,v)\geq \left(f^\ast \circ a_\phi \right)^{-1}(t), \ \ \textrm{for all } (x,v) \in S_1(t),$$
and, from $\mu_g\leq \mu_f$, we have for all $t\in \R^+$, $meas(S_1(t)) \leq meas(S_2(t))$. Thus, since $\left(f^\ast \circ a_\phi \right)^{-1}\leq 0$,
$$T \geq \int_0^{+\infty} dt \int_{S_2(t)} \left( \left(f^\ast \circ a_\phi \right)^{-1}(t)-e(x,v) \right) dxdv. $$
But,
$$e(x,v)< \left(f^\ast \circ a_\phi \right)^{-1}(t), \ \ \textrm{for all } (x,v) \in S_2(t),$$
which implies on the one hand that $T\geq 0$ and in the other hand that, if $T=0$, we have, for almost all $t\in \R_+$, $meas(S_2(t))=0$. Thus we have, almost everywhere $meas(S_1(t))=meas(S_2(t))=0$  which gives $g=f^{\ast\phi}$.

\end{proof}
Now we give a lemma in which states a method to inverse clearly the function $f^\ast \circ a_\phi$.

\begin{lemma}[Pseudo inverse of $f^\ast \circ a_\phi$]\label{lem-pseudo} Let $f\in \mathcal{E}_p$, nonzero, with $p>\frac{3}{2}$ and $\phi\in \Phi_q$. Let $f^\ast$ the Schwartz rearrangement of $f$ in $\R^6$. We define the pseudo inverse of $f^\ast \circ a_\phi$ for $s\in(0,\| f \|_{L^\infty})$ as
\be (f^\ast \circ a_\phi)^{-1} (s)=\sup\{ e \in (\inf \phi, 0) \ : \ f^\ast \circ a_\phi (e)>s \}. \label{def-pseudo} \ee
Then $(f^\ast \circ a_\phi)^{-1} $ is a nonincreasing function from $(0,\| f \|_{L^\infty})$ to $(\inf \phi, 0)$ and for all $(x,v)\in\RR^6$ and $s\in(0,\| f \|_{L^\infty})$,
\be f^{\ast\phi}(x,v)>s \Rightarrow \sqrt{|v|^2+1} -1 +\phi(x) \leq  (f^\ast \circ a_\phi)^{-1} (s), \label{inverse1} \ee
and
\be f^{\ast\phi}(x,v)\leq s \Rightarrow \sqrt{|v|^2+1} -1 +\phi(x) \geq  (f^\ast \circ a_\phi)^{-1} (s). \label{inverse2} \ee
\end{lemma}

\bs
\ni
For the proof of this lemma, we refer to \cite{LMR-inv}.

\end{document}